\documentclass{amsart}

\oddsidemargin \evensidemargin
\usepackage[latin1]{inputenc}
\usepackage{subfigure}

\usepackage{amscd,amssymb,amsmath,tikz}

\newtheorem{theorem}{Theorem}[section]

\newtheorem{lemma}[theorem]{Lemma}
\newtheorem{cor}[theorem]{Corollary}
\theoremstyle{definition}
\newtheorem{definition}[theorem]{Definition}

\numberwithin{equation}{section}
\numberwithin{figure}{section}
\usepackage{color}
\usepackage{amsmath, latexsym, amssymb}
\usepackage{graphicx}
\usepackage{amsfonts}
\usepackage{subfigure}
\usepackage{graphics}
\usepackage{latexsym}
\usepackage{color}
\usepackage{graphicx}
\usepackage{pict2e}
\usepackage[all]{xy}
\usepackage{enumerate}
\usepackage{enumitem}
\usepackage{verbatim}

\newcommand{\comp}{\boldsymbol{\beta}}

\newlength{\cellsize}
\newcommand\tableau[1]{
\vcenter{
\let\\=\cr
\baselineskip=-16000pt
\lineskiplimit=16000pt
\lineskip=0pt
\halign{&\tableaucell{##}\cr#1\crcr}}}

\newcommand{\tableaucell}[1]{{%
\def \arg{#1}\def \void{}%
\ifx \void \arg
\vbox to \cellsize{\vfil \hrule width \cellsize height 0pt}%
\else
\unitlength=\cellsize
\begin{picture}(1,1)
\put(0,.22){\makebox(1,1)[b]{$#1$}}
\put(0,0){\line(1,0){1}}
\put(0,1){\line(1,0){1}}
\put(0,0){\line(0,1){1}}
\put(1,0){\line(0,1){1}}
\end{picture}%
\fi}}

\author{Sarah K. Mason
  \and Elizabeth Niese$^*$
}
\thanks{$^*$ This work was supported by the Wake Forest University Pilot Collaboration Grant.}
\title{Quasisymmetric $(k,l)$-hook Schur functions}

\begin{document}

\begin{abstract}

We introduce a quasisymmetric generalization of Berele and Regev's hook Schur functions and prove that these new quasisymmetric hook Schur functions decompose the hook Schur functions in a natural way.  We examine the combinatorics of the quasisymmetric hook Schur functions, providing a relationship to Gessel's fundamental quasisymmetric functions and an analogue of the Robinson-Schensted-Knuth algorithm.  We also prove that the multiplication of quasisymmetric hook Schur functions with hook Schur functions behaves the same as the multiplication of quasisymmetric Schur functions with Schur functions.
  \end{abstract}

\maketitle
\maketitle
\section{Introduction}

Hook Young diagrams became of interest as the classical Schur-Weyl duality was extended to the general linear Lie superalgebra.  Schur~\cite{Sch01} determined a one-to-one correspondence between irreducible representations of the general linear group $GL(V)$ and subsets of the irreducible representations of $S_n$.  Weyl's {\em Strip Theorem}~\cite{Wey39} states that these irreducible representations of the general linear group $GL(V)$ are precisely those obtained from partitions whose Young diagrams lie inside a strip of height $k$, where $k$ is the dimension of the vector space $V$.  Schur's action of $S_n$ on $V^{\otimes n}$ and that from Weyl's Strip Theorem are dual.  Berele and Regev~\cite{BerReg87} generalize the two actions of $S_n$ on $V^{\otimes n}$ into a single action by considering a decomposable vector space $V=T\oplus U$ such that $\dim(T)=k$ and $\dim(U)=l$.  In this new setting, the indexing set is given by partitions which lie inside a hook shape of height $k$ and width $l$, called a {\em hook Young diagram}, meaning there are at most $k$ parts greater than $l$.  Berele and Regev use certain fillings of these diagrams to generate polynomials known as $(k,l)$-hook Schur functions on two sets of variables, which generalize the classical Schur functions and appear naturally when examining characters of a certain $S_n$ representation of $GL(k) \times GL(l)$.

Remmel~\cite{Rem84} introduces an analogue of the Robinson-Schensted-Knuth (RSK) algorithm for {\it $(k,l)$-semistandard tableaux}, the objects used to generate $(k,l)$-hook Schur functions.  The insertion algorithm underlying this RSK analogue is an important component in the rule for multiplying two $(k,l)$-hook Schur functions.  This procedure is similar to the Schensted insertion used in the proof of the Littlewood-Richardson rule for the product of two Schur functions.  Remmel~\cite{Rem87} further exploits the rich structure of these objects to prove a number of permutation statistic identities, including a generalization of the {\it Cauchy Identity}~\cite{Mac92} which provides a generating function for products of Schur functions.

The Schur functions (which form a basis for symmetric functions) can be obtained as specializations of Macdonald polynomials~\cite{Mac95}.  A basis for quasisymmetric functions can be similarly obtained through specializations of nonsymmetric Macdonald polynomials~\cite{HLMvW09}.  This basis, called the {\it quasisymmetric Schur function basis}, can also be obtained by summing certain collections of Type A Demazure atoms and is of interest due to its combinatorial similarities to the Schur functions~\cite{LMvW13} as well as its algebraic significance in the noncommutative character theory of the symmetric group~\cite{vWi13}.

In this paper, we provide a quasisymmetric analogue of the $(k,l)$-hook Schur functions obtained by summing the weights of fillings of composition diagrams satisfying certain conditions and prove that this analogue decomposes the $(k,l)$-hook Schur functions in a natural way.\footnote{An extended abstract of the first three sections and the fifth and sixth sections of this paper appeared in \cite{MasNie14}.}  In Section~\ref{background}, we describe the hook composition tableaux used to generate the quasisymmetric hook Schur functions which are obtained from a combination of quasisymmetric Schur functions and row-strict quasisymmetric Schur functions.  In Section~\ref{section:properties} we discuss several properties of the quasisymmetric hook Schur functions.  Section~\ref{section:fun} describes a relationship between the quasisymmetric hook Schur functions and a superized version of Gessel's fundamental quasisymmetric functions.  In Sections~\ref{section:insertion} and ~\ref{section:rsk} we introduce an insertion algorithm and use it to provide an analogue of the Robinson-Schensted-Knuth algorithm as well as a generalized Cauchy identity.  In Section~\ref{sec:LR} we prove a Littlewood-Richardson rule for the product of a quasisymmetric hook Schur function and a hook Schur function.  Sections~\ref{background},~\ref{section:properties}, \ref{section:insertion}, and ~\ref{section:rsk} appeared in an extended abstract format in~\cite{MasNie14} but Sections~\ref{section:fun} and~\ref{sec:LR} are new to this paper.

\subsection{Acknowledgements}

The authors would like to especially thank Jeff Remmel for suggesting this avenue of research and his helpful comments along the way.  The helpful comments of an anonymous reviewer were also instrumental in completing the proof of the Littlewood-Richardson analogue.

\section{Background}{\label{background}}

The $(k,l)$-hook Schur functions introduced by Berele and Regev~\cite{BerReg87} are defined combinatorially using $(k,l)$-semistandard hook tableaux.  Frequently the $k$ and $l$ designations are dropped and these diagrams are referred to simply as {\it semistandard hook tableaux}.  Begin with the {\it Young diagram} of $\lambda = (\lambda_1, \lambda_2, \hdots , \lambda_n)$, which is given by placing $\lambda_i$ boxes (or {\it cells}) in the $i^{th}$ row from the bottom of the diagram, in French notation.  A {\it $(k,l)$-semistandard hook tableau} (SSHT) of shape $\lambda$ is a filling of the Young diagram of $\lambda$ with letters from two different alphabets $\mathcal{A}=\{1,2,\ldots, k\} $ and $\mathcal{A}'=\{1',2',\ldots, l'\}$ where $k<1'$ such that the unprimed entries weakly increase from left to right along rows and strictly increase from bottom to top in columns while the primed entries strictly increase from left to right along rows and weakly increase from bottom to top in columns~\cite{BerReg87}.  All rows and columns must be weakly increasing, so in any given column all of the primed entries appear in a higher row than all of the unprimed entries, and in any given row all of the primed entries appear to the right of all of the unprimed entries, as seen in Figure~\ref{(k,l)SSYT}.  We can then define a $(k,l)${\em-hook Schur function} as a polynomial in $X=x_1,x_2,\ldots,x_k$ and $Y=y_1,y_2,\ldots, y_l$ by
\[HS_\lambda (X;Y) = \sum_{T\in SSHT(\lambda)} x_1^{v_1}x_2^{v_2}\cdots x_k^{v_k}y_1^{u_1}y_2^{u_2}\cdots y_l^{u_l},\]
where $v_i$ is the number of times $i$ occurs in $T$ and $u_i$ is the number of times $i'$ occurs in $T$.  Note that $HS_\lambda(X;0)=s_\lambda(X)$ and $HS_\lambda(0;Y)=s_{\lambda'}(Y)$ where $\lambda'$ is the conjugate partition of $\lambda$ and $s_\lambda$ is a Schur polynomial.

\begin{figure}
$$T=\tableau{ 1' \\ 1' & 2'  \\ 2 & 2 & 3' & 4'  \\ 1 & 1 & 4 & 4'}$$
\caption{$T$ is a semistandard hook tableau of shape $(4,4,2,1)$ and weight $x_1^2 x_2^2 x_4 y_1^2y_2y_3y_4^2$.}\label{(k,l)SSYT}
\end{figure}

The quasisymmetric Schur functions were introduced in \cite{HLMvW09} as polynomials generated by {\it composition tableaux}, generalizations of semistandard Young tableaux whose underlying shapes are compositions instead of partitions.  These polynomials form a basis for quasisymmetric functions and decompose the Schur functions in a natural way.  A closely related set of polynomials, still given by fillings of composition diagrams~\cite{LMvW13}, is more natural for us to work with for the purposes of this paper, but note that this new form, denoted here by $\mathcal{CS}_{\alpha}$, is easily obtained from the original definition by a reversal of the entries in a filling.  A slight modification of the definition of $\mathcal{CS}_{\alpha}$ produces a new basis for quasisymmetric functions that is generated using a row-strict analogue of the composition tableaux~\cite{MasRem10}.  We will again work with the variation, $\mathcal{RS}_{\alpha}$, of the row-strict quasisymmetric functions obtained by the same reversal procedure as is employed in \cite{LMvW13}.  In fact, we further extend this approach to include skew compositions as indexing compositions as in \cite{MN-SkewRS}.

Let $\alpha = (\alpha_1, \alpha_2, \hdots , \alpha_m)$ be a composition.  Then its {\it composition diagram} is given by placing $\alpha_i$ boxes (or {\it cells}) in the $i^{th}$ row from the bottom of the diagram, in French notation.  The cells are indexed by row and column, so that $(i,j)$ refers to the cell in the $i^{th}$ row from the bottom and the $j^{th}$ column from the left. The composition $\alpha$ is said to be {\it contained} in the composition $\beta$ (written $\alpha \subset \beta$) if and only if $\ell(\alpha) \le \ell(\beta)$ and $\alpha_i \le \beta_i$ for all $1 \le i \le \ell(\alpha)$.  
 Given $\alpha \subset \beta$, we define a {\it skew composition diagram} of shape $\beta//\alpha$ to be the diagram of $\beta$ with the cells of $\alpha$ removed from the bottom left corner.

\begin{definition}\label{SSYCTdef}
A filling $T: \beta // \alpha \rightarrow \mathbb{Z}_+$ is a {\it semi-standard Young composition tableau (SSYCT)} of shape $\beta // \alpha$ if it satisfies the following conditions:

\begin{enumerate}
\item the row entries are weakly increasing from left to right,
\item the entries in the leftmost column are strictly increasing from bottom to top, and 
\item (triple condition) for $1\leq i<j \leq \ell(\beta)$ and $1\leq k<m$, where $m$ is the size of the largest part of $\beta$, if $T(i,k+1)\neq \infty$ and $T(i,k+1) \ge T(j,k)$, then $T(i,k+1) >T(j,k+1)$, assuming the entry in any cell not contained in $\beta$ is $\infty$ and the entry in any cell contained in $\alpha$ is 0. 
\end{enumerate}

Let $\beta$ be a composition.  Then the Young quasisymmetric Schur function $\mathcal{CS}_{\beta}$ is given by $$\mathcal{CS}_{\beta}=\sum_T x^T$$ where the sum is over all SSYCT $T$ of shape $\beta$, where $\beta$ may be a skew composition.  

\end{definition}

\begin{definition}\label{SSYRTdef}
A filling $T: \beta // \alpha \rightarrow \mathbb{Z}_+$ is a {\it semi-standard Young row-strict composition tableau (SSYRT)} of shape $\beta // \alpha$ if it satisfies the following conditions:

\begin{enumerate}
\item the row entries are strictly increasing from left to right,
\item the entries in the leftmost column are weakly increasing from bottom to top, and 
\item (triple condition) for $1\leq i<j \leq \ell(\beta)$ and $1\leq k<m$, where $m$ is the size of the largest part of $\beta$, if $T(i,k+1)\neq \infty$ and $T(i,k+1)>T(j,k)$, then $T(i,k+1)\geq T(j,k+1)$, assuming the entry in any cell not contained in $\beta$ is $\infty$ and the entry in any cell contained in $\alpha$ is 0. 
\end{enumerate}

Let $\beta$ be a composition.  Then the Young row-strict quasisymmetric Schur function $\mathcal{RS}_{\beta}$ is given by $$\mathcal{RS}_{\beta}=\sum_T x^T$$ where the sum is over all SSYRT $T$ of shape $\beta$, where $\beta$ may be a skew composition.  

\end{definition}

Combining these two approaches, we have the following definition for the composition analogue of a $(k,l)$-semistandard hook tableau.
\begin{definition}\label{def:HCT}
As before, let $\mathcal{A}=\{1,2,\ldots, k\}$ and $\mathcal{A}'=\{1',2',\ldots, l'\}$ with total ordering $1<2<\cdots<k<1'<2'\cdots<l'$ on $\mathcal{A}\cup\mathcal{A}'$.
Given a composition diagram $\alpha=(\alpha_1, \alpha_2, \hdots , \alpha_r)$ with largest part $m$, a {\it hook composition tableau} (HCT), $F$, is a filling of the cells of $\alpha$ with letters from $\mathcal{A}\cup\mathcal{A'}$ such that
\begin{enumerate}
\item the entries of $F$ weakly increase in each row when read from left to right,
\item the unprimed (resp. primed) entries of $F$ weakly (resp. strictly) increase in each row 
when read from left to right,
\item the unprimed (resp. primed) entries in the leftmost column of $F$ strictly (resp. weakly) increase when read from bottom to top, 
\item and $F$ satisfies the following {\it triple rule}:\\
Supplement $F$ by adding enough cells with infinity-valued entries to the end of each row so that the resulting supplemented tableau, $\hat{F}$, is of rectangular shape $r \times m$.  Then for $1 \le i<j \le r, \; \; 2 \le n \le m,$ where $\hat{F}(i,n)$ denotes the entry of $\hat{F}$ that lies in the cell in the $i$-th row from the bottom and $n$-th column from the left,
\begin{enumerate}[label=(\alph*)]
\item if $\hat{F}(i, n+1) \in \mathcal{A}$ and $\hat{F}(i,n+1) \ge \hat{F}(j,n),$ then $\hat{F}(i,n+1) > \hat{F}(j,n+1),$ and
\item if $\hat{F}(i, n+1) \in \mathcal{A}'$ and $\hat{F}(i,n+1) > \hat{F}(j,n),$ then $\hat{F}(i,n+1) \ge \hat{F}(j,n+1).$
\end{enumerate}
\end{enumerate}
\end{definition}

See Figure~\ref{fig:triple} for the triple configuration.
Note that triple rule (a) is identical to the triple rule used to define a {\it semi-standard Young composition tableau} in Definition~\ref{SSYCTdef}
 and that triple rule (b) is identical to the triple rule used to define a {\it row-strict semi-standard Young composition tableau} in Definition~\ref{SSYRTdef}.  
This is due to the fact that the unprimed portion of the filling behaves like a semistandard Young composition tableau while the primed portion behaves like a row-strict analogue of a skew semistandard Young composition tableau.
\begin{figure}
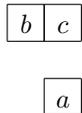

  $$\tableau{b & c \\ \\ & a}.$$
\caption{Triple configuration with $\hat{F}(i,n+1)=a$, $\hat{F}(j,n)=b$, and $\hat{F}(j,n+1)=c$.  If $a \in \mathcal{A}$, then $a \ge b \Rightarrow a > c$.  If $a \in \mathcal{A'}$, then $a>b \Rightarrow a \ge c$.}\label{fig:triple}
\end{figure}{\label{fig:triple}}

\begin{definition}
The {\it quasisymmetric $(k,l)$-hook Schur function $\mathcal{HQ}_{\alpha}$
} indexed by the composition $\alpha$ is given by $$ \mathcal{HQ}_{\alpha} (X; Y)=\sum_{F \in HCT(\alpha)}x_1^{v_1}x_2^{v_2} \cdots x_k^{v_k}y_1^{u_1}y_2^{u_2} \cdots y_l^{u_l},$$ where $HCT(\alpha)$ is the set of all hook composition tableaux of shape $\alpha$, $v_i$ is the number of times the letter $i$ appears in $F$, and $u_i$ is the number of times the letter $i'$ appears in $F$.
\end{definition}


See Figure~\ref{hookQS} for an example of a quasisymmetric $(k,l)$-hook Schur function and the fillings appearing in such a function.  Note that $\mathcal{CS}_\alpha(X) = \mathcal{HQ}_\alpha(X;0)$ and $\mathcal{RS}_\alpha(X) = \mathcal{HQ}_\alpha(0;X)$.

\begin{figure}[b]
$$\mathcal{HQ}_{(1,2,1)}(x_1,x_2; y_1, y_2)=$$
$$\tableau{1' \\ 2 & 2 \\ 1} \quad \tableau{1' \\ 2 & 1' \\1} \quad \tableau{2' \\ 2 & 2 \\ 1} \quad \tableau{2' \\ 2 & 1' \\ 1} \quad \tableau{2' \\ 2 & 2' \\ 1} \quad  \tableau{2' \\ 1' & 2' \\ 1} \quad \tableau{2' \\ 1' & 2' \\ 2} \quad \tableau{2' \\ 1' & 2' \\ 1'}$$
$ x_1 x_2^2 y_1 \;    +   \; x_1 x_2 y_1^2 \;  + \;  x_1 x_2^2 y_2   +   x_1 x_2 y_1 y_2   + x_1 x_2 y_2^2 \;  +  \;  x_1 y_1 y_2^2 \;  +  \;  x_2 y_1 y_2^2 \; +  \; y_1^2 y_2 ^2 \quad $
\caption{The quasisymmetric $(2,2)$-hook Schur function $\mathcal{HQ}_{(1,2,1)}(x_1,x_2; y_1, y_2)$.}\label{hookQS}
\end{figure}

\section{Properties of the quasisymmetric $(k,l)$-hook Schur functions}{\label{section:properties}}

Every Schur function decomposes into a positive sum of quasisymmetric Schur functions~\cite{HLMvW09}.  Similarly, every Schur function also decomposes into a positive sum of row-strict quasisymmetric Schur functions~\cite{MasRem10, MN-SkewRS}.  In particular, $$s_{\lambda} = \sum_{\boldsymbol{\lambda}(\alpha) = \lambda} \mathcal{CS}_{\alpha} = \sum_{\boldsymbol{\lambda}(\alpha) = \lambda'} \mathcal{RS}_{\alpha}$$
The following theorem demonstrates the fact that this behavior continues as expected in the case of quasisymmetric $(k,l)$-hook Schur functions.

\begin{theorem}{\label{thm:hookSchur}}
The $(k,l)$-hook Schur functions decompose into a positive sum of quasisymmetric $(k,l)$-hook Schur functions in the following way:  $$HS_{\lambda} (X; Y) = \sum_{\boldsymbol{\lambda}(\alpha)=\lambda} \mathcal{HQ}_{\alpha} (X; Y),$$
where $\boldsymbol{\lambda}(\alpha)=\lambda$ indicates that the parts of $\alpha$ rearrange to $\lambda$ when placed in weakly decreasing order.
\end{theorem}

\begin{proof}
We exhibit a weight-preserving bijection, $f$, between the set of all semistandard hook tableaux of shape $\lambda$ and the set of all hook composition tableaux whose shape rearranges to $\lambda$.  This map is a generalization of the map given in~\cite{HLMvW09} between semistandard tableaux and composition tableaux.

Given a semistandard hook tableau $T$ of shape $\lambda$, map the entries in the leftmost column of $T$ to the leftmost column of $f(T)$ by placing them in weakly increasing order from bottom to top.  Map each remaining set of column entries from $T$ into the corresponding column of $f(T)$ by the following process.  
\begin{enumerate}
\item Assume that the entries in the first $j-1$ columns have been inserted into $f(T)$ and begin with the smallest entry, $a_1$, in the set of entries in the $j^{th}$ column of $T$.
\item If $a_1$ is unprimed, map $a_1$ to the highest available cell that is immediately to the right of an entry weakly smaller than $a_1$.  If $a_1$ is primed, map $a_1$ to the highest available cell that is immediately to the right of an entry strictly smaller than $a_1$.
\item Repeat Step 2 with the next smallest entry, noting that a cell is {\it available} if no entry has already been placed in this cell.
\item Continue until all entries from this column have been placed, and then repeat with each of the remaining columns.
\end{enumerate} 

See Figure \ref{TheMap} for an example of this map.  We must show that this process produces a hook composition tableau.  Suppose $F=f(T)$ for some hook tableau $T$.  The first three conditions are satisfied by construction, so we must check the fourth (triple) condition.  For part (a), consider two cells $\hat{F}(i,k+1)$ and $\hat{F}(j,k)$ such that $\hat{F}(i,k+1) \in \mathcal{A}$ and $\hat{F}(i,k+1) \ge \hat{F}(j,k)$.  We must show that $\hat{F}(i,k+1) > \hat{F}(j,k+1).$  Let $\hat{F}(j,k)=b$, $\hat{F}(i,k+1)=a$, and $\hat{F}(j,k+1)=c$.  Then the cells are situated as shown in Figure~\ref{fig:triple} with $a \ge b$.

We must prove that $\hat{F}(i,k+1) > \hat{F}(j,k+1)$, or in other words, that $a > c$.  Assume, to get a contradiction, that $a < c$.  (We know $c \not= a$ since $a \in \mathcal{A}$ and there are no repeated column entries from $\mathcal{A}$. ) Since $a < c$, then $a$ would be inserted into its column before $c$.  But then the cell immediately to the right of $b$ would be available during the insertion of $a$, and therefore $a$ would be placed in that cell since $a \ge b$ and $a \in \mathcal{A}$.  Therefore this configuration would not occur and thus $a>c$.  

Next, for part (b), consider two cells $\hat{F}(i,k+1)$ and $\hat{F}(j,k)$ such that $\hat{F}(i,k+1) \in \mathcal{A}'$ and $\hat{F}(i,k+1) > \hat{F}(j,k)$.  We must show that $\hat{F}(i,k+1) \ge \hat{F}(j,k+1).$  Let $\hat{F}(j,k)=b$, $\hat{F}(i,k+1)=a$, and $\hat{F}(j,k+1)=c$ as before.  Then the cells are situated as in Figure~\ref{fig:triple} with $a > b$.

We must prove that $\hat{F}(j,k+1) \ge \hat{F}(i,k+1)$, or in other words, that $a \ge c$.  Assume, to get a contradiction, that $a < c$.  Then, as before, $a$ would be inserted into its column before $c$.  But then the cell immediately to the right of $b$ would be available during the insertion of $a$, and therefore $a$ would be placed in that cell since $a > b$ and $a \in \mathcal{A}'$.  Therefore this configuration would not occur and thus $a \ge c$.  


The inverse map, $f^{-1}$, is given by arranging the entries from each column of a hook composition tableau $U$ so that the unprimed entries are strictly increasing from bottom to top, and above them the primed entries are weakly increasing from bottom to top.  Since repeated unprimed entries within a column of a hook composition tableau would not satisfy the triple rule, the unprimed entries in the diagram $f^{-1}(U)$ must be strictly increasing from bottom to top.  We must therefore prove that if two entries, $x$ and $y$, are in the same row  of $f^{-1}(U)$ with $x$ immediately to the left of $y$, then $x\leq y$ with strict inequality if $x \in \mathcal{A}'$.  Argue by contradiction.  Assume first that there exists a row in $f^{-1}(U)$ in which $x$ is immediately to the left of $y$ but $x > y$.  Choose the leftmost column $c$ in $f^{-1}(U)$ in which such an $x$ exists, and the lowest row $r$ containing this situation with $x$ in column $c$.  Then column $c$ contains only $r-1$ entries which are less than or equal to $y$ while column $c+1$ contains $r$ entries less than or equal to $y$.  Since the column entries in the hook composition tableau $U$ are the same as the column entries in $f^{-1}(U)$, this implies that one of the entries less than or equal to $y$ in column $c+1$ of the hook composition tableau must lie immediately to the right of an entry that is greater than $y$, which contradicts the definition of a hook composition tableau.  


Next assume there exists a row in $f^{-1}(U)$ in which $x$ is immediately to the left of $y$ and $x = y$ but $x=y$ is primed.  Again, select the leftmost column $c$ in $f^{-1}(U)$ containing such an $x$, and the lowest row $r$ containing this situation.  Again, the column $c$ contains only $r-1$ entries which are less than $y$ while column $c+1$ contains $r$ entries less than or equal to $y$.  Since the column entries in the hook composition tableau $U$ are the same as the column entries in $f^{-1}(U)$, this implies that one of the entries less than or equal to $y$ in column $c+1$ of the hook composition tableau must lie immediately to the right of an entry that is greater than or equal to $y$, which contradicts the definition of a hook composition tableau.  
\end{proof}

\begin{figure}
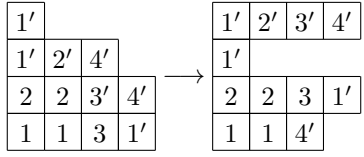

$$ \tableau{1' \\ 1' & 2' & 4' \\ 2 & 2 & 3' & 4' \\ 1 & 1 & 3 & 1' } \longrightarrow \tableau{1' & 2' & 3' & 4' \\ 1' \\ 2 & 2 & 3 & 1' \\ 1 & 1 & 4'}$$
\caption{The bijection $f$ maps a semistandard hook tableau of shape $(4,4,3,1)$ to a hook composition tableau of shape $(3,4,1,4)$.}
\label{TheMap}
\end{figure}

Notice that each hook composition tableau appearing in a given quasisymmetric hook Schur function can be broken into its row-strict portion and its column-strict portion.  We may therefore decompose each quasisymmetric hook Schur function into a sum of products of quasisymmetric Schur functions and skew row-strict quasisymmetric Schur functions as follows: $$\mathcal{HQ}_{\alpha}(X;Y) = \sum_{\beta \subseteq \alpha} \mathcal{CS}_{\beta} (X) \mathcal{RS}_{\alpha // \beta} (Y).$$

This is analogous to the decomposition of the hook Schur functions into sums of products of Schur functions and skew Schur functions \cite{BerRem85} given by  $$HS_{\lambda} (X; Y) = \sum_{\mu \subseteq \alpha} s_{\mu}(X) s_{\lambda' / \mu'} (Y).$$

However, some other quasisymmetric analogies of straightforward results about hook Schur functions do not carry through as directly.  For example, one can see that \begin{equation}\label{transpose} HS_{\lambda} (X;Y) = HS_{\lambda'} (Y;X)\end{equation} by taking the transpose of each generating semistandard hook tableau \cite{BerRem85}.  However, taking the transpose of a composition (which rearranges to a partition $\lambda$) using the standard method (writing it as a ribbon and then transposing the ribbon and recording the underlying composition) does not produce a composition which rearranges the transpose of the partition $\lambda$ as seen in Figure~\ref{fig:transpose}.  However, a related result can be obtained once we define a ``row-strict'' analogue to hook composition tableaux.

\begin{figure}
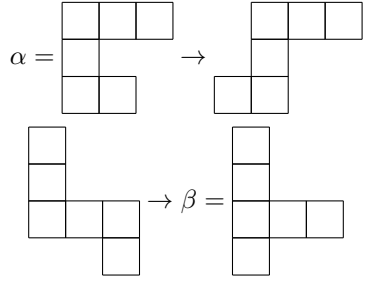

\[\alpha =  \tableau{{}&{}&{}\\{}\\{}&{}}  \rightarrow  \tableau{&{}&{}&{}\\&{}\\{}&{}} \]
\[ \tableau{{}\\{}\\{}&{}&{}\\&&{}} \rightarrow \beta = \tableau{{}\\{}\\{}&{}&{}\\{}}\]
\caption{The composition $\alpha = (2,1,3)$ has underlying partition $\lambda=(3,2,1)$.  The transpose of $\alpha$ obtained by writing as a ribbon, transposing, and identifying the underlying composition is $\beta=(1,3,1,1)$.  The underlying partition of $\beta$ is $(3,1,1,1)\neq \lambda'$.}\label{fig:transpose}
\end{figure}

\begin{definition}\label{def:rowstrictHCT}
As before, let $\mathcal{A}=\{1,2,\ldots, k\}$ and $\mathcal{A}'=\{1',2',\ldots, l'\}$ with total ordering $1<2<\cdots<k<1'<2'\cdots<l'$ on $\mathcal{A}\cup\mathcal{A}'$.  Given a composition diagram $\alpha=(\alpha_1, \alpha_2, \hdots , \alpha_r)$ with largest part $m$, a {\it row-strict hook composition tableau} (RHCT), $F$, is a filling of the cells of $\alpha$ with letters from $\mathcal{A}\cup\mathcal{A'}$ such that
\begin{enumerate}
\item the entries of $F$ weakly increase in each row when read from left to right,
\item the unprimed (resp. primed) entries of $F$ strictly (resp. weakly) increase in each row 
when read from left to right,
\item the unprimed (resp. primed) entries in the leftmost column of $F$ weakly (resp. strictly) increase when read from bottom to top, 
\item and $F$ satisfies the following {\it triple rule}:

Supplement $F$ by adding enough cells with infinity-valued entries to the end of each row so that the resulting supplemented tableau, $\hat{F}$, is of rectangular shape $r \times m$.  Then for $1 \le i < j \le r, \; \; 2 \le n \le m,$ where $\hat{F}(i,n)$ denotes the entry of $\hat{F}$ that lies 
in the cell in the $i$-th row from the bottom and $n$-th column from the left,
\begin{enumerate}[label=(\alph*)]
\item if $\hat{F}(i, n+1) \in \mathcal{A}'$ and $\hat{F}(i,n+1) \ge \hat{F}(j,n),$ then $\hat{F}(i,n+1) > \hat{F}(j,n+1),$ and
\item if $\hat{F}(i, n+1) \in \mathcal{A}$ and $\hat{F}(i,n+1) > \hat{F}(j,n),$ then $\hat{F}(i,n+1) \ge \hat{F}(j,n+1).$
\end{enumerate}
\end{enumerate}
\end{definition}

Note that this definition switches the role of the primed alphabet with that of the unprimed alphabet.  Let $RHCT(\alpha)$ denote the set of row-strict hook composition tableaux of shape $\alpha$.  Then we can define a {\em row-strict quasisymmetric $(k,l)$-hook Schur function} as 
\[\mathcal{RHQ}_\alpha (X;Y) = \sum_{F\in RHCT(\alpha)} x_1^{v_1}\cdots x_k^{v_k} y_1^{u_1}\cdots y_l^{u_l}\] where $v_i$ is the number of times $i$ appears in $F$ and $u_i$ is the number of times $i'$ appears in $F$.  

\begin{theorem}\label{thm:transpose}
For $\lambda \vdash n$, 
\[\sum_{\boldsymbol{\lambda}(\alpha) = \lambda} \mathcal{HQ}_\alpha(X;Y) = \sum_{\boldsymbol{\lambda}(\beta) = \lambda'} \mathcal{RHQ}_\beta (X;Y).\]
\end{theorem}

Before we prove Theorem~\ref{thm:transpose}, we note that the left hand side is equal to $\mathcal{HS}_{\lambda}(X;Y)$ by Theorem~\ref{thm:hookSchur}.  Theorem~\ref{thm:transpose} together with Equation~\ref{transpose} therefore imply that the right hand side is equal to $\mathcal{HS}_{\lambda'}(Y;X)$.  The bijective proof of Theorem~\ref{thm:transpose} given below provides a direct method for taking the transpose of a hook composition tableau.

\begin{proof}{(of Theorem~\ref{thm:transpose})} We define a map $\hat{\phi}:HCT(\alpha: \boldsymbol{\lambda}(\alpha)= \lambda) \rightarrow RHCT(\beta: \boldsymbol{\lambda}(\beta) = \lambda')$.  This map is analogous to the map $\phi$ in \cite{MasRem10} taking column-strict composition tableaux to row-strict composition tableaux.  Given $F\in HCT(\alpha)$, construct $\hat{\phi}(F)$ by first taking the smallest entry in each column of $F$ and placing the entries in the first column of $\hat{\phi}(F)$ from bottom to top, smallest to largest.  Then take the next smallest entry in each column of $F$ and place these entries, smallest to largest, in the highest row of $\hat{\phi}(F)$ satisfying the row increasing conditions (i.e., strict for unprimed entries, weak for primed entries).  Continue likewise until all entries have been used.  The result of $\hat{\phi}$ can be observed in Fig.~\ref{fig:phi}.  

To see that $\hat{\phi}(F)$ is a RHCT, note that the row-increasing rules are met by construction, as is the increasing first column.  It remains to show that the triple rules are satisfied.  


Consider a triple of entries $a=\hat{\phi}(F)(i,k+1), b=\hat{\phi}(F)(j,k),$ and $c=\hat{\phi}(F)(j,k+1)$ with $i<j$, as depicted in Figure~\ref{fig:triple}. First suppose $a\in \mathcal{A}$ and $a>b$.  Since $a>b$, but $a$ was not placed in cell $(j,k+1)$, $c$ must have been placed prior to $a$, and thus $a\geq c$.  Next, suppose $a\in \mathcal{A}'$ and $a\geq b$.  Since $a\geq b$ and was not placed in cell $(j,k+1)$, $c$ must have been placed earlier than $a$.  Thus, $c\leq a$.  But, if $c = a$, then in $F$, $c$ and $a$ were both the $k+1$st smallest entries in their respective columns.  But, since $F$ is row-strict in the primed entries, there are no repeated primed entries that are both the $k+1$st smallest entries in their respective columns.  Thus, $a>c$ and both triple rules are preserved.

We can define $\hat{\phi}^{-1}$ on an RHCT $T$ by taking the smallest entry in each column of $T$ and placing the entries in the first column of $\hat{\phi}^{-1}(T)$ from bottom to top, smallest to largest.  Then take the next smallest entry in each column of $T$ and place them into the second column of $\hat{\phi}^{-1}(T)$, from smallest to largest, in the highest row of $\hat{\phi}^{-1}(T)$ satisfying the row increasing conditions, weakly increasing for unprimed entries and strictly increasing for primed entries.  Continue likewise until all entries have been used.  The proof that the resulting diagram satisfies the triple conditions is completely analogous to the proof above.
\end{proof}

\begin{figure}
\[F = \tableau{1'&2'&3'\\4&4&1'&2'\\3&2'\\2&2&3&1'\\1&1&1} \qquad \hat{\phi}(F) = \tableau{1'&2'\\1&2&3&4&1'\\1&2&4&2'&2'\\1&3&1'&3'}\]
\caption{The result of applying $\hat{\phi}$ to an HCT.}\label{fig:phi}
\end{figure}

\section{Super fundamental quasisymmetric functions}\label{section:fun}

Both the column-strict and row-strict quasisymmetric functions can be decomposed as sums of {\em fundamental} quasisymmetric functions as shown in \cite{LMvW13,MN-SkewRS}.  Let $D\subseteq [n-1]$.  Then the fundamental quasisymmetric function indexed by $D$ is 
\[F_{n,D} (X) = \sum_{\substack{a_1\leq a_2\leq \cdots \leq a_n\\ a_i = a_{i+1} \Rightarrow i \notin D}} x_{a_1}x_{a_2}\cdots x_{a_n}.\]

To write $\mathcal{CS}_\alpha$ as a sum of fundamental quasisymmetric functions, we first need to define a standardization map for SSYCT.  Given a SSYCT $T$, we standardize by reading up each column from the left column to the right, replacing the first 1 with 1, the second 1 with 2, etc., then replacing the first 2 with $\alpha_1+1$, continuing likewise.  The result is $std(T)$, a {\em standard Young composition tableau} (SYCT); that is, a SSYCT with $n$ cells containing the entries $1, \ldots,n$ each used exactly once.  The {\em descent set} of a SYCT $F$ is $D(F)=\{i:i+1$ is weakly left of $i$ in $F\}$.  Then, for $\alpha \vDash n$, 
\begin{equation}\label{eq:colfund}
\mathcal{CS}_\alpha (X) = \sum_{T \in SYCT(\alpha)} F_{n, D(T)}.\end{equation}

Standardization for SSYRT works similarly.  Given a SSYRT $T$, standardize by reading up each column from right to left, replacing the first 1 with 1, the second with 2, etc., continuing as before with this reading order.  The result is $st(T)$, a {\em standard Young row-strict composition tableau} (SYRT), that is, a SSYRT with $n$ cells containing the entries $1,\ldots,n$ each used exactly once.  The {\em descent set} of a SYRT $G$ is $\hat{D}(G) = \{i:i+1$ is strictly right of $i$ in $G\}$.  Then, for $\alpha \vDash n$, 
\begin{equation}\label{eq:rowfund}
\mathcal{RS}_\alpha(X) = \sum_{T \in SYRT(\alpha)} F_{n, \hat{D}(T)}.\end{equation}

In \cite{HHLRU} a {\em super fundamental quasisymmetric function} is introduced.  
\[\tilde{Q}_{n,D}(X,Y) = \sum_{\substack{a_1\leq a_2\leq \cdots \leq a_n\\ a_i=a_{i+1} \in \mathcal{A} \Rightarrow i \notin D\\ a_i = a_{i+1} \in \mathcal{A}' \Rightarrow i \in D}} z_{a_1}z_{a_2}\cdots z_{a_n}\]
where $z_a = x_a$ for $a \in \mathcal{A}$ and $z_{a'} = y_a$ for $a' \in \mathcal{A}'$.  In this section we show that the quasisymmetric $(k,l)$-hook Schur functions can be decomposed into the super fundamental quasisymmetric functions when the alphabet $\mathcal{A} = 1<2<\hdots< k < \mathcal{A}' = 1'<2'<\hdots<l'$ and appropriate indexing sets are used as considered in \cite{Kwo09}.  

We can write the $\tilde{Q}_{n,D}$ in terms of the fundamental quasisymmetric functions.  First we define the {\em composition corresponding to a set $S$}, denoted $\comp(S)$ (where $S=\{s_1,s_2,\hdots,s_k\}\subseteq \{1, \hdots, n-1\}$), given by $\comp(S)=(s_1,s_2-s_1,\hdots,s_k-s_{k-1},n-s_k)$.  Note that $\comp(S)$ is a composition of $n$.  Given two compositions $\alpha=(\alpha_1,\hdots,\alpha_k)$ and $\beta=(\beta_1,\hdots,\beta_m)$, the {\em concatenation} of $\alpha$ and $\beta$ is given by $\alpha \cdot \beta = (\alpha_1,\hdots,\alpha_k,\beta_1,\hdots,\beta_m)$, while the {\em almost concatenation} of $\alpha$ and $\beta$ is given by $\alpha\odot\beta=(\alpha_1,\hdots,\alpha_{k-1},\alpha_k+\beta_1, \beta_2,\hdots,\beta_m)$.  In the following, $\overline{D}$ is the set complement of $D$.

\begin{theorem}\label{lem:superfund}
Let $D\subseteq [n-1]$.  Then 
\begin{equation}\label{eq:superfund}
\tilde{Q}_{n,D}(X,Y) = \sum_{i=0}^n F_{i,D_1}(X) F_{n-i, D_2}(Y)
\end{equation}
where $\boldsymbol{\beta}(D_1)\cdot \boldsymbol{\beta}(\overline{D_2}) = \boldsymbol{\beta}(D)$ if $i \in D$ and $\boldsymbol{\beta}(D_1) \odot \boldsymbol{\beta}(\overline{D_2}) = \boldsymbol{\beta}(D)$ if $i \notin D$.
\end{theorem}

\begin{proof}
Each monomial in $\tilde{Q}_{n,D}$ has the form $x_{a_1}x_{a_2}\cdots x_{a_k}y_{b_1}\cdots y_{b_{n-k}}$ where $a_1\leq a_2\leq \cdots\leq a_k$, $b_1\leq b_2\leq\cdots\leq b_{n-k}$.  Note that $x_{a_i}=x_{a_{i+1}}$ implies $i \notin D$ and $y_{b_j}=y_{b_{j+1}}$ implies $j + k \in D$, which occurs when $j+k \notin \overline{D}$.  We fix a $k$ and show that the monomials with degree $k$ in the $x$'s are precisely those appearing in $F_{i,D_1}(X) F_{n-i, D_2}(Y)$ for certain pairs of sets $D_1, D_2$ described below. 

Suppose $k \in D$.  Then $D = \{d_1, d_2, \ldots, d_i=k, \ldots, d_m\} \subseteq [n-1]$.  Let $D_1 = \{d_1, \ldots, d_{i-1}\} \subseteq [k-1]$ and $D_2 = [n-k-1] \setminus \{d_{i+1}-k, \ldots, d_m-k\}$, so $\overline{D_2} = \{d_{i+1}-k, \ldots, d_m - k\}$.  Then 
\[\boldsymbol{\beta}(D) = (d_1, d_2-d_1, \ldots, k-d_{i-1}, d_{i+1}-k, \ldots, d_m-d_{m-1}, n-d_m),\]
 while \[\boldsymbol{\beta}(D_1) = (d_1, d_2-d_1, \ldots, d_{i-1}-d_{i-2}, k-d_{i-1})\] and \begin{align*} \boldsymbol{\beta}(\overline{D_2}) &= (d_{i+1}-k,( d_{i+2}-k)-(d_{i+1}-k), \ldots, n-k-(d_m-k))\\& = (d_{i+1}-k, d_{i+2}-d_{i+1}, \ldots, d_m-d_{m-1}, n-d_m).\end{align*}
Thus, $\boldsymbol{\beta}(D) = \boldsymbol{\beta}(D_1)\cdot \boldsymbol{\beta}(\overline{D_2})$.

Now, suppose $k \notin D$.  Then $D = \{d_1, d_2, \ldots, d_m\}$ and there is some $i$ such that $d_i<k<d_{i+1}$.  Set $D_1 = \{d_1, \ldots, d_i\} \subseteq [k-1]$ and $D_2 = [n-k-1]\setminus \{d_{i+1} - k, \ldots, d_m-k\}$, so $\overline{D_2}=\{d_{i+1} - k,\ldots, d_m - k\}$.  Then 
\[\boldsymbol{\beta}(D) = (d_1, d_2-d_1, \ldots, d_i-d_{i-1}, d_{i+1}-d_i, \ldots, d_m-d_{m-1}, n-d_m),\]
while
\[ \boldsymbol{\beta}(D_1) = (d_1, d_2-d_1, \ldots, d_i-d_{i-1}, k-d_i)\] and
\[\boldsymbol{\beta}(\overline{D_2}) = (d_{i+1}-k, d_{i+2}-d_{i+1}, \ldots, d_m-d_{m-1}, n-d_m).\]
Thus, 
\begin{align*}\boldsymbol{\beta}(D) &= (d_1, d_2-d_1, \ldots, (k-d_i)+(d_{i+1}-k), \ldots, n-d_m)\\
& = \boldsymbol{\beta}(D_1) \odot \boldsymbol{\beta}(\overline{D_2}).\qedhere\end{align*} 
\end{proof}

To {\em standardize} a HCT with content $(\alpha_1,\alpha_2,\ldots, \alpha_k)$ in the unprimed entries and $(\beta_1,\beta_2,\ldots,\beta_l)$ in the primed entries, start first with the unprimed entries.  Reading up each column, from left to right, replace the first 1 with 1, the second 1 in this reading order with 2, etc., then replace the first 2 with $\alpha_1+1$, and continue likewise.  In the primed entries, read up each column, right to left, replacing the first $1'$ with $\sum_i a_i +1$, the second $1'$ with $\sum_i a_i+2$, etc., then replacing the first $2'$ with $\sum_i a_i+\beta_1+1$, and so on.  This process can be seen in Fig.~\ref{fig:superstdz}.  The {\em descent set} for a standard composition filling $T$ is $D(T) = \{ i \in [n-1]: i+1 \text{ is weakly left of } i \}$.  Denote the set of {\em standard hook composition tableaux} of shape $\alpha$ by $SHCT(\alpha)$ and denote the standardization of an HCT $S$ by $stdz(S)$.  Note that these are equivalent to standard composition tableaux and satisfy the triple rules from Definition~\ref{def:HCT}.  

\begin{figure}
\[\begin{array}{cc}
T = \tableau{1'&2'&3'\\4&4&1'&2'\\3&2'\\2&2&3&1'\\1&1&1} & stdz(T) = \tableau{12&15&16\\8&9&11&13\\6&14\\4&5&7&10\\1&2&3}\\\\
& D(stdz(T)) = \{3,5,7,10,11,13,14\}\end{array}
\]
\caption{Standardization of HCT} \label{fig:superstdz}
\end{figure}

\begin{theorem}\label{thm:qhookfund} For $\alpha \vDash n$, 
\[\mathcal{HQ}_\alpha (X;Y) = \sum_{T \in SHCT(\alpha)} \tilde{Q}_{n, D(T)}(X;Y).\]
\end{theorem}
\begin{proof}
We proceed by showing that the set of monomials in $\mathcal{HQ}_\alpha(X;Y)$ is the same as the set of monomials in $\sum_{T\in SHCT(\alpha)} \tilde{Q}_{n,D(T)}(X;Y)$.  First, suppose that $x_{a_1}\cdots x_{a_k} y_{b_1}\cdots y_{b_l}$ (where $k+l =n$) is the content monomial associated with a HCT $S$ of shape $\alpha$ with subscripts arranged in weakly increasing order.  Let $T = stdz(S)$.  To show that the content monomial appears in $\tilde{Q}_{n,D(T)}(X;Y)$, we must show that it satisfies the two conditions: if $a_i=a_{i+1}$ then $i \notin D(T)$ and if $b_i=b_{i+1}$, then $i+k \in D(T)$.  Both follow almost immediately from the standardization procedure.  First note that if $a_i=a_{i+1}$ then $a_i$ and $a_{i+1}$ are in distinct columns of $S$ since $S$ is column-strict in the unprimed entries.  Thus, when standardizing, $i+1$ is placed in a column strictly to the right of $i$.  Thus $i \notin D(T)$.  Similarly, if $b_i = b_{i+1}$, then during standardization, $i+k+1$ will appear weakly left of $i+k$ following the rules for standardizing primed entries.  Thus, $i+k \in D(T)$.  Therefore, 
$x_{a_1}\cdots x_{a_k} y_{b_1}\cdots y_{b_l}$ is in $\tilde{Q}_{n,D(T)}(X;Y)$.

Now, suppose $x_{a_1}\cdots x_{a_k} y_{b_1}\cdots y_{b_l}$ is a monomial in $\tilde{Q}_{n,D(T)}(X;Y)$.  To find a HCT $S$ of shape $\alpha$, we reverse the standardization map.  That is, replace $n$ by $b_l$, $n-1$ by $b_{l-1}$, etc.  It remains to show that the result is a hook composition tableau.  Note that the condition that $b_i=b_{i+1}$ implies $i+k \in D(T)$ means that, when computing the HCT $S$, each row will contain distinct primed entries, increasing from left to right.  Similarly, the condition that if $a_i=a_{i+1}$ implies $i \notin D(T)$ guarantees that column entries are distinct.  Note that the rows increase weakly by construction.  The triple conditions are also satisfied by this reverse map.  In $T$, given the triple configuration in Figure~\ref{fig:triple}, if $T(a)>T(b)$, then $T(a)>T(c)$ since each entry of $T$ is distinct.  Then, for the corresponding triple in $S$, if $S(a) \in \mathcal{A}$, we know that $S(a)\geq S(b)$ and also $S(a) \geq S(c)$ by the reverse standardization map.  Since the reverse standardization map preserves the requirement that column entries are distinct among unprimed entries, $S(a) \neq S(c)$, and the first triple condition is satisfied.  Similarly, if $S(a) \in \mathcal{A}'$, we know that $S(a) \geq S(b)$ and $S(a)\geq S(c)$ by the reverse standardization map.  Thus the second triple condition is satisfied. Thus $x_{a_1}\cdots x_{a_k}y_{b_1}\cdots y_{b_l}$ is a monomial in $\mathcal{HQ}_\alpha(X;Y)$.
\end{proof}

\section{An insertion algorithm for hook composition tableaux}\label{section:insertion}

We give an analogue of the composition tableau insertion algorithm \cite{Mas08} for hook composition tableaux.  Note that this algorithm also gives insertion algorithms for column- and row-strict composition tableaux if restricted to just one alphabet.    

Given a hook composition tableau $F$ and $x \in \mathcal{A}\cup \mathcal{A}'$, we insert $x$ into $F$, denoted $F\leftarrow x$, in the following way:
\begin{enumerate}
\item Read down each column of $\hat{F}$, starting from the rightmost column and moving left. This is the {\em reading order} for $\hat{F}$.
\begin{enumerate}
\item If $x\in \mathcal{A}$, replace with $x$ the first entry $\hat{F}(i,j)$ such that $\hat{F}(i,j)>x$, $\hat{F}(i,j-1)\leq x$, and $j\neq 1$.  We say that ``$x$ bumps $\hat{F}(i,j)$''.  If there is no such entry, then insert $x$ into a new row of length 1 in the first column, in between the unique pair $\hat{F}(i,1)$ and $\hat{F}(i+1,1)$ such that $\hat{F}(i,1)<x<\hat{F}(i+1,1)$.  If $x<\hat{F}(1,1)$, insert $x$ in a new cell at the bottom of the leftmost column. 
\item If $x\in \mathcal{A}'$, bump the first entry $\hat{F}(i,j)$ such that $\hat{F}(i,j)\geq x$, $\hat{F}(i,j-1)<x$, and $j \neq 1$.  If there is no such entry, then insert $x$ into a new row of length 1 in the first column, in between the unique pair $\hat{F}(i,1)$ and $\hat{F}(i+1,1)$ such that $\hat{F}(i,1)<x\leq \hat{F}(i+1,1)$.  If $x\leq \hat{F}(i,1)$ for all $i$, insert $x$ in new cell at the bottom of the first column.  

\end{enumerate}

\item If $\hat{F}(i,j)=\infty$ or a new row was created, the insertion terminates.  Otherwise, set $x=\hat{F}(i,j)$ and continue to scan cell entries in reading order, starting at the cell immediately following $(i,j)$ in reading order.  

\item Continue likewise until the insertion terminates.
\end{enumerate}

In Figure \ref{fig:inserthook} we show the insertion algorithm for several valus of $x\in \mathcal{A}\cup\mathcal{A}'$.

\begin{figure}[h]
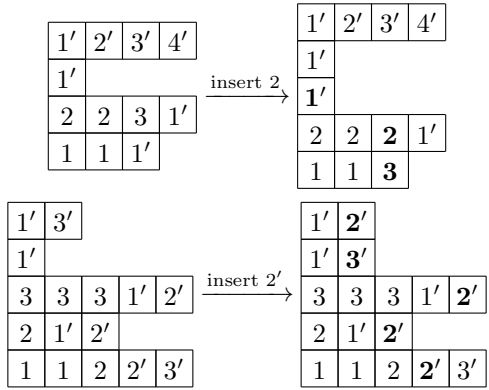

\[\tableau{1'&2'&3'&4'\\1'\\2&2&3&1'\\1&1&1'} \xrightarrow{\text{insert }2} \tableau{1'&2'&3'&4'\\1'\\\mathbf{1'}\\2&2&\mathbf{2}&1'\\1&1&\mathbf{3}}\]\[ \tableau{1'&3'\\1'\\3&3&3&1'&2'\\2&1'&2'\\1&1&2&2'&3'} \xrightarrow{\text{insert }2'} \tableau{1'&\mathbf{2'}\\1'&\mathbf{3'}\\3&3&3&1'&\mathbf{2'}\\2&1'&\mathbf{2'}\\1&1&2&\mathbf{2'}&3'}\]

\caption{Row-insertion into a hook composition tableau with bumping paths in bold.}\label{fig:inserthook}
\end{figure}

We show that the insertion algorithm yields a hook composition tableau.  We also prove that the insertion algorithms for semistandard hook tableaux and hook composition tableaux commute with the function $f$ from Theorem~\ref{TheMap}.

\begin{lemma}\label{lem:equal} Let $F$ be a hook composition tableau and let $c_1=(i_1,j_1)$ and $c_2=(i_2,j_2)$ be cells in $F$ such that $(i_1,j_1)$ appears before $(i_2,j_2)$ in reading order, $F(c_1)=F(c_2)=a$, and no cell between $c_1$ and $c_2$ in reading order has label $a$.  In $F\leftarrow k$, let $(\overline{i_1}',\overline{j_1})$ and $(\overline{i_2},\overline{j_2})$ be the cells containing the entries $F(c_1)$ and $F(c_2)$ (respectively) after the insertion.  Then $(\overline{i_1},\overline{j_1})$ appears before $(\overline{i_2},\overline{j_2})$ in reading order and if $a \in \mathcal{A}$ then $\overline{j_1} > \overline{j_2}$.
\end{lemma}
 

\begin{proof}
Suppose $F$ is a hook composition tableau, $c_1=(i_1,j_1)$ and $c_2=(i_2,j_2)$ with $c_1$ appearing before $c_2$ in reading order, $F(c_1)=F(c_2)=a$, and no cell between $c_1$ and $c_2$ in reading order has label $a$.  Consider $F\leftarrow k$.  Note that if $k>a$, neither $F(c_1)$ nor $F(c_2)$ will be bumped since the sequence of bumped entries is weakly increasing.  So, suppose $k\leq a$. 

If $a\in \mathcal{A}'$, and $F(c_1)$ is bumped, then either $F(c_1)$ will bump some entry from a cell $d$ between $c_1$ and $c_2$ in reading order or $F(c_1)$ bumps $F(c_2)$ by insertion rule (1)(b).  Thus $(\overline{i_1},\overline{j_1})$ occurs prior to $(\overline{i_2},\overline{j_2})$ in reading order.

If $a \in \mathcal{A}$ it follows that $j_1>j_2$ since all unprimed entries in a column must be distinct.  We must look at several cases.

First, consider $j_1-j_2>1$.  Then, if $F(c_1)$ does not bump an entry prior to $F(i_2,j_2+1)$, $F(c_1)$ will be inserted in cell $(i_2,j_2+1)$ by insertion rule (1)(a).  So, in this case, $(\overline{i_1},\overline{j_1})$ occurs prior to $(\overline{i_2},\overline{j_2})$.  

If $j_2+1=j_1$, then note that $i_2\leq i_1$ since, if $i_2>i_1$, then by the triple rule, $F(i_2,j_1)<F(c_1)$, but then row $i_2$ would not be increasing.  If $i_2=i_1$, then $F(c_1)$ cannot be bumped, though it is possible that $F(c_2)$ could be.  If $i_2<i_1$ and $F(c_1)$ is bumped, then $F(c_1)$ will bump $F(i_2,j_2+1)$ or an entry earlier in reading order.  In either case, $(\overline{i_1},\overline{j_1})$ will occur earlier in the reading order than $(\overline{i_2},\overline{j_2})$ and $\overline{j_1}>\overline{j_2}$.
\end{proof}

\begin{lemma} \label{insertHCT} The result of the insertion algorithm is a hook composition tableau.\end{lemma}

\begin{proof}  Suppose the sequence of labels bumped during the insertion $F\leftarrow x$ is $x_0=x, x_1,\ldots, x_m$ with each $x_i$ bumped from cell $c_i$ for $1\leq i \leq m$.  We prove by induction that the result of each bump is a hook composition tableau.  Note that conditions (1), (2), and (3) are satisfied by construction.  Therefore it suffices to show that all triples continue to satisfy the triple rule throughout the insertion process. 

If the insertion algorithm results in a new row being created we need only consider triples of the form 
\[\tableau{x_m&\infty\\\\&a},\] which means that we must show $a<x_m$ if $a\in \mathcal{A}$ and $a\leq x_m$ if $a\in \mathcal{A}'$.  Suppose that $a\in \mathcal{A}$ and $a\geq x_m$.  Then it must be that the label $x_m$ was bumped from some cell between $c_m$ and the cell containing $a$ in reading order, else $x_m$ would have bumped $a$ during the insertion process since $a\geq x_m$ and $x_m$ must be greater than the entry in the cell left of $a$, which we label $\overline{a}$.  But, in this case, prior to $x_m$ being bumped from cell $c_m$, either there was a triple of the form \[\tableau{\overline{a}&a\\\\&x_m},\] which violates the condition that $F$ is a HCT, or cell $c_m$ is in the first column of $F$, which implies $x_m\in\mathcal{A}'$ and hence $x_m>a$.  The case with $a\in \mathcal{A}'$ is similar.

Any triple of cells not involving the bumped entry will be unaffected by the insertion, so it is sufficient to consider only those triples involving the bumped entry.  Assuming that $x_0, x_1, \ldots, x_{j-1}$ have been placed by the insertion algorithm, consider what occurs when $x_j$ bumps $x_{j+1}$ from cell $c_{j+1}$.  We suppose that $x_j , x_{j+1}\in \mathcal{A}$, noting that the cases where $x_j \in \mathcal{A}$ with $x_{j+1} \in \mathcal{A}'$ and  $x_j,x_{j+1} \in \mathcal{A}'$ are similar.  When $x_j$ is inserted, bumping $x_{j+1}$ from cell $c_{j+1}$, we consider three possible locations of the cell $c_{j+1}$ in the triple $$\tableau{b&c\\\\&a}.$$  First we consider when $c_{j+1}$ is in position $a$.  Note that it is impossible for $F(b)\leq x_j<F(c)$ since if this were the case, $x_j$ would have been bumped from a cell above $c$ by the triple rule and hence would end up bumping $F(c)$ instead of $x_{j+1}$.  Thus, the triple condition is satisfied.

Next we consider when $x_j$ bumps $x_{j+1}$ from position $b$ in the triple.  We must show that if $x_j \le F(a)$, then $F(a) > F(c)$, so assume $x_j \le F(a)$.  Note that by Lemma \ref{lem:equal}, $x_j$ cannot be equal to $F(a)$ so in fact we may assume that $x_j < F(a)$.  For any cell $d$, let $\overline{d}$ indicate the cell directly to the left of $d$.  Given cells arranged as in the diagram 
\[ \tableau{\overline{b}& b & c\\\\&\overline{a}&a},\] recall that $x_j$ bumps the entry $x_{j+1}$ from position $b$.  Since the triple condition was satisfied prior to inserting $x_j$ either $x_{j+1} \le F(a)$ and $F(a) > F(c)$ or $x_{j+1} > F(a)$.  If  $x_{j+1} \le F(a)$ and $F(a) > F(c)$, then we are done, so assume $x_{j+1} > F(a)$.  Then  $x_{j+1}>F(\overline{a})$, since $F(a) \ge F(\overline{a})$.  Therefore, by the triple condition, $F(\overline{b})>F(\overline{a})$.  Then, since $x_j \ge F(\overline{b})$ and $F(\overline{b}) > F(\overline{a})$, we have $x_j>F(\overline{a})$.  Since $F(\overline{a})<x_j < F(a)$, if $x_j$ was bumped from a cell earlier in the reading order than $a$, then $x_j$ would bump $F(a)$ rather than $x_{j+1}$.  Thus $x_j$ must be bumped from a cell between $a$ and $b$ in the reading order.    Since $F(\overline{a})<x_j<F(a)$, $x_j$ could not have been bumped from a cell in the same column and below $a$ or else the triple condition would not have been satisfied.  Thus, $x_j$ must have been bumped from the same column as $b$.  However, in this case, since $x_j>F(\overline{a})$, we know that $F(\overline{c_j})>F(\overline{a})$ since the triple condition was satisfied prior to $x_j$ being bumped from cell $c_j$, so $x_{j-1}\geq F(\overline{c_j})>F(\overline{a})$, thus the triple condition is satisfied after $x_j$ is bumped from cell $c_j$.  Similarly, $x_i>F(\overline{a})$ for all $0\leq i\leq j$ and each of $x_1,x_2,\ldots, x_j$ is bumped from the same column.  But, since $x_0>F(\overline{a})$ and $x_0<F(a)$, $x_0$ must have bumped $F(a)$ instead of $x_1$.  Therefore $x_j<F(c)<F(a)$ and the triple condition is satisfied. 

Finally, if $x_j$ bumps $x_{j+1}$ from position $c$ in the triple, then $F(b)\leq x_j<x_{j+1}$.  Thus, since the triple condition was satisfied in $F$, we know that either $F(a)<F(b)$ or $x_{j+1}<F(a)$.  In either case, the triple condition is satisfied after the insertion of $x_j$.  
\end{proof}

\begin{definition}[Remmel's insertion algorithm \cite{Rem84}]\label{def:reminsert}  To insert $x$ into a semistandard hook tableau $T$, denoted $T\Leftarrow x$, do the following:
\begin{enumerate}
\item Insert $x$ into the first row by bumping the entry $y=T(i,j)$ with the property that, if $x \in \mathcal{A}$ (resp. $x \in \mathcal{A}'$), $T(i,j-1)\leq x <T(i,j)$ (resp. $T(i,j-1)<x\leq T(i,j)$).  If $x$ is greater than every entry in the row, append $x$ in a new cell at the end of the row and the algorithm terminates.
\item Insert the entry $y$ into the second row, following the same process as before.  Continue until the algorithm terminates.

\end{enumerate}
\end{definition}

\begin{lemma}{\label{lem:commute}}
The insertion procedure  $T \Leftarrow x$ commutes with the bijection $f$ from semistandard hook tableaux to hook composition tableaux used in the proof of Theorem \ref{thm:hookSchur}.  That is, given a semistandard hook tableau $T$ and $x\in \mathcal{A}\cup \mathcal{A}'$, $f(T\Leftarrow x) = f(T)\leftarrow x$ where $T\Leftarrow x$ is computed using the row-insertion algorithm in Definition~\ref{def:reminsert} to insert an entry into a semistandard hook tableau.
\end{lemma}

\begin{proof}

Consider an arbitrary semistandard hook tableau $T$ and $x \in \mathcal{A} \cup \mathcal{A'}$.  We wish to prove that $f(T\Leftarrow x) = f(T)\leftarrow x$.  Since $f$ preserves column entries, we need only show that the columns of $f(T) \leftarrow x$ contain the same entries as the columns of $T \Leftarrow x$.  The column entries of $T$ are the same as the column entries of $f(T)$ by the argument given in the proof of Theorem~\ref{thm:hookSchur}.  Suppose the insertion sequence for $T \Leftarrow x$ is $x \le x_1 \le x_2 \le \cdots \le x_m$.  (Note that if no bumps occur then $x=x_m$ and $m=0$ giving an insertion sequence consisting of just $x$.)  Two situations are possible; either the insertion terminates in the same column which contained the first bumped entry or at some point during the insertion of $x$ into $T$ an entry $y$ is bumped from column $i$ to some column $j$ such that $j<i$.

\vspace*{.1in}

\noindent {\bf Case 1.}  The insertion terminates in the same column which contained the first bumped entry.   (See Figure~\ref{fig:onecol} for an illustration of the situation.)

\begin{figure}
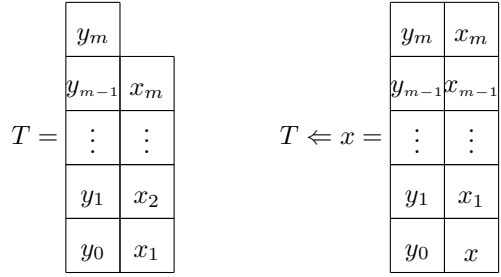

$$T ={\cellsize=4.75ex \tableau{ y_m \\ y_{\scriptscriptstyle m-1} & x_m \\ \vdots & \vdots \\ y_1 & x_2 \\ y_0 & x_1}} \qquad \qquad T \Leftarrow x = {\cellsize=4.75ex\tableau{y_m & x_m \\ y_{\scriptscriptstyle m-1} & x_{\scriptscriptstyle m-1} \\ \vdots & \vdots \\ y_1 & x_1 \\ y_0 & x}}$$
\caption{Only columns $i$ and $i-1$ are shown, where column $i-1$ might contain more entries above $y_m$.}
\label{fig:onecol}
\end{figure}

Let column $i$ be the column in which the insertion terminates.  First assume that all elements of the insertion sequence are in $\mathcal{A}$.  Then in column $i-1$ of $T$ there are at least $m+1$ entries $y_0, y_1, \hdots, y_m$ such that $x_k \ge y_k \ge y_{k-1} \ge \cdots \ge y_0$ for each $k$, and $x \ge y_0$.  Note that $x$ is less than all entries in column $i+1$ of $T$ (and also $f(T)$) since otherwise $x$ would bump an entry before reaching column $i$.  We also have $x$ is less than all entries in column $i$ of $T$ (and $f(T)$) since $x$ bumps the smallest entry in column $i$.  Since $x \ge y_0$, There is at least one entry in column $i-1$ of $T$ (and therefore $f(T)$) which is less than or equal to $x$.  Therefore $x$ will bump something in column $i$ of $f(T)$ or be placed into column $i$ to terminate the procedure.  (Note that it doesn't necessarily bump $x_1$; we just know that it bumps some entry.)

We need to show that when $x_k$ in column $i$ of $f(T)$ is bumped there exists a location in column $i$ of $f(T)$ such that either $x_k$ can be placed in a new cell at this location or $x_k$ bumps the entry $x_j$ contained in this cell in $f(T)$.

Suppose row $p$ is the highest row of $f(T)$ in which the entry in cell $(i-1,p)$ is less than or equal to $x$.  If cell $(i,p)$ is empty, then $x$ is placed into this cell and the process terminates as desired.  Otherwise, $x$ will bump the entry $x_k$ in cell $(i,p)$ and the insertion process continues.

Recall that when $f(T)$ is constructed from $T$, each entry $x_k$ is placed in the highest available location in column $i$ starting with $x_1$ and working from smallest to largest.  Since for each $k$ there are $k+1$ entries in column $i-1$ which are less than or equal to $x_k$ and $k-1$ entries from column $i$ of $T$ have been placed before placing $x_k$, there must always be at least one entry in column $i-1$ which is less than or equal to $x_k$ appearing in a lower row that $x_k$ in $f(T)$.  Therefore when any $x_k$ is bumped during insertion, there is a valid location for $x_k$ to be placed in column $i$ of $f(T)$.

If at some point during the insertion procedure $T \Leftarrow x$ an entry $x_j \in \mathcal{A}'$ is bumped, then in column $i-1$ of $T$ there are still at least $m+1$ entries $y_0, y_1, \hdots, y_m$ in column $i-1$ of $T$ such that $x_k \ge y_k \ge y_{k-1} \ge \cdots \ge y_0$ for each $k$ but there are also $m-j+1$ entries in column $i-1$ of $T$ such that $x_{\ell} > y_{\ell} > y_{k\ell1} > \cdots > y_j$ for all $\ell \ge j$.  The remainder of the  argument for the $\mathcal{A}$ case goes through identically.  Similarly, if all elements of the insertion sequence are in $\mathcal{A}'$, the statement becomes the following.  In column $i-1$ of $T$ there are still at least $m+1$ entries $y_0, y_1, \hdots, y_m$ in column $i-1$ of $T$ such that $x_k > y_k > y_{k-1} > \cdots > y_0$ for each $k$ and the proof procedes as above.

\vspace*{.1in}

\noindent {\bf Case 2.}  At some point during the insertion of $x$ into $T$ an entry $y$ is bumped from column $i$ to some column $j$ such that $j<i$.  (See Figure~\ref{fig:skipcols} for an illustration of the situation.)

\begin{figure}
$$T = \tableau{ z & b & b & {} \\ {} &  s & s & y \\ {} & s  & s & {} } \qquad \qquad T \Leftarrow x = \tableau{y & b & b & {} \\ {} & s & s & x_k \\ {} & s & s & {}}$$
\caption{The column containing $z$ in $T$ is column $j$, the column containing $y$ in $T$ is column $i$, the entries $s$ must be smaller than $y$, and the entries $b$ must be bigger than $y$. Other cells in these columns are omitted, as are other columns.}
\label{fig:skipcols}
\end{figure}

Assume that $y \in \mathcal{A}$.  (We will omit the argument for $y \in \mathcal{A'}$ since the proof is essentially the same; see the proof of case 1 for an indication of how the inequalities change to modify the proof in this situation.)  We first establish that during the bumping process on $f(T)$ if $y$ is bumped from column $i$ of $f(T)$ it cannot bump something else in column $i$ of $f(T)$ and must instead scan the next column.  To see this, suppose $y$ is in the cell $(i,p)$ of $T$.  Then there are exactly $p$ entries in column $i-1$ of $T$ that are less than or equal to $y$.

The mapping $f$ from $T$ to $f(T)$ sent $y$ to the highest position in column $i$ such that the entry to its left was less than or equal to $y$ that did not already contain an entry less than $y$.  Since there are exactly $p$ entries in column $i-1$ of $f(T)$ which are less than or equal to $y$, and $y$ is the $p^{th}$ largest entry in column $i$ of $f(T)$, once $y$ is inserted every entry less than or equal to $y$ in column $i-1$ of $T$ will have an entry less than or equal to $y$ immediately to its right.  This means in particular that any entry in column $i-1$ in a row below $y$ that contains an entry less than or equal to $y$ must have an entry to its right that is strictly less than $y$.  Therefore if $y$ is bumped during the insertion procedure $f(T) \leftarrow x$ then it cannot be placed or inserted into column $i$.

Next we show there is no place between columns $i$ and $j$ for $y$ to either bump another entry or be placed in a new cell.  To see this, first note that there are exactly $p$ entries less than or equal to $y$ in $T$ in each column between column $j$ and column $i$.  (If there were more entries less than or equal to $y$ in any of these columns, say column $k$, $y$ would have been placed in column $k+1$ instead of in column $j$.)

Thus, when $y$ is bumped from column $i$ during $f(T) \leftarrow x$, there will be no location for $y$ to bump another entry or be placed in a new cell until column $j$, at which point there is guaranteed to be a valid location for $y$ to be placed or bump another entry.

It remains to show that if $y$ is the entry bumped from column $i$ in $T$ then $y$ is the entry bumped from column $i$ in $f(T)$.   Let $x_0$ be the entry that is first bumped into column $i$ of $T$ during $x \Leftarrow T$.  We may inductively assume that $x_0$ is the entry that is first bumped into column $i$ of $f(T)$ during $x \leftarrow f(T)$.  Let $x_0 \le x_1 \le x_2 \le \cdots \le x_k \le y$ be the portion of the insertion sequence of $x \Leftarrow T$ contained in column $i$ of $T$.

If $y$ is not bumped during $x \leftarrow f(T)$, then $x_{\ell}$ for some $\ell$ such that $1 \le \ell \le k$ is the largest entry in column $i$ of $f(T)$ that is bumped.  (Note that no entry larger than $y$ can be bumped since the entries larger than $y$ are all located immediately to the right of entries larger than $y$.)  Column $i-1$ must contain at least $\ell+1$ entries less than or equal to $\ell$, but since $x_{\ell}$ is the $\ell^{th}$ largest entry in column $i$ only $\ell-1$ entries were placed in column $i$ of $f(T)$ during the map from $T$ to $f(T)$, so there must at least one row below the row containing $x_{\ell}$ that contains an entry less than or equal to $x_{\ell}$ in column $i-1$.  Therefore $x_{\ell}$ must bump the entry to the right of this entry less than or equal to $x_{\ell}$, and therefore $y$ is indeed bumped from column $i$ during the insertion procedure $x \leftarrow f(T)$.   \qedhere

\end{proof}

\section{An analogue of the Robinson-Schensted-Knuth Algorithm}{\label{section:rsk}}

The Robinson-Schensted-Knuth (RSK) Algorithm is a bijection between matrices with non-negative integer entries and pairs of semi-standard Young tableaux.  This correspondence utilizes an intermediary step which sends the matrix to a two-line array (a bi-word) satisfying certain properties and can be used to obtain information about the bottom line of this array, such as the length of the longest increasing subsequence in this word.  When restricted to permutation matrices, the Robinson-Schensted-Knuth algorithm provides an elegant proof that the number of pairs of standard Young tableaux with $n$ cells is equal to $n!$. 

Berele and Remmel~\cite{BerRem85} extend this algorithm to a bijection between members of a certain class of matrices and pairs of semistandard hook tableaux.  They use this to prove several important identities for Hook Schur functions including the following analogue of the Cauchy identity:

\begin{align}
\sum_{\lambda} HS_{\lambda}( X;S)& HS_{\lambda}(Y;T)\nonumber\\
& = \prod_{i,j} (\frac{1}{1-x_i y_j}) \prod_{i,j}(\frac{1}{1 - s_i t_j}) \prod_{i,j} (1+x_i t_j) \prod_{i,j} ( 1 + y_i s_j).{\label{cauchy}}
\end{align}

We use the insertion procedure described in Section~\ref{section:insertion} to introduce an analogous bijection between members of a certain class of matrices and pairs of hook composition tableaux.  This algorithm can be used to prove a generating function identity for quasisymmetric hook Schur functions.


The following definition describes the class of matrices used by Berele and Remmel~\cite{BerRem85}, which is also the class that we use in our variation of the Robinson-Schensted-Knuth algorithm.  Intuitively, the blocks in this class of matrices can be thought of as the portions of the diagrams which allow repeated entries (signified by nonnegative integers greater than one) and those portions which do not allow repeated entries.

\begin{definition}
Let $M$ be a $(k_2 + l_2) \times (k_1 + l_1)$ matrix.  Then $M \in \mathcal{M} (k_1, l_1, k_2, l_2)$ if and only if $M$ satisfies the conditions given in the following diagram:

\begin{center}
\begin{picture}(150,100)
\put(0,0){\line(0,1){100}}
\put(0,50){\line(1,0){150}}
\put(75,0){\line(0,1){100}}
\put(150,0){\line(0,1){100}}
\put(20,85){$k_2 \times k_1$}
\put(10,70){nonnegative}
\put(15, 60){integers}
\put(95,30){$l_2 \times l_1$}
\put(85,15){nonnegative}
\put(95, 5){integers}
\put(20,30){$l_2 \times k_1$}
\put(20, 10){$0$ or $1$}
\put(95,85){$k_2 \times l_1$}
\put(95, 65){$0$ or $1$}
\end{picture}
\end{center}
\end{definition}

Berele and Remmel give a bijection from this set of matrices to pairs of semistandard hook tableaux to prove the following theorem. 

\begin{theorem}{\label{BRrsk}}{\cite{BerRem85}}
There exists a weight preserving bijection between matrices in the collection $\mathcal{M} (k_1, l_1, k_2, l_2)$ and pairs of semistandard hook tableaux of the same shape.
\end{theorem}

The Theorem below provides a bijection from this same set of matrices to a different collection of tableau diagrams; namely to pairs of hook composition tableaux instead of pairs of semistandard hook tableaux.  The process is similar to that of Berele and Remmel; the hook composition tableaux diagrams appear because we use the insertion process described in Section 4 rather than the insertion given by Berele and Remmel.

\begin{theorem}{\label{rsk}}
There exists a weight preserving bijection between matrices in the collection $\mathcal{M} (k_1, l_1, k_2, l_2)$ and pairs of hook composition tableaux with the same underlying partition.
\end{theorem}

\begin{proof}
Let $M$ be a matrix in $\mathcal{M} (k_1, l_1, k_2, l_2)$.  We form a word $M(\omega)$ comprised of biletters ${ i \choose j}$ following the method of Berele and Remmel~\cite{BerRem85}.  Label the columns of $M$ by $1,2, \hdots , k_1, 1', 2', \hdots, l'_1$ from left to right and label the rows of $M$ by $1,2, \hdots, k_2, 1', 2', \hdots , l'_2$.  Each entry $M_{i,j}$ of $M$ is associated to a biletter ${ p \choose q}$ where $p$ is the label of the $i^{th}$ row of $M$ and $Q$ is the label of the $j^{th}$ column of $M$.  The biletters within the word $M(\omega)$ are arranged so that the top elements are weakly increasing (with primed elements considered greater than unprimed elements) and the biletters with equal top elements are arranged according to the following.  (See Figure~\ref{fig:RSK} for an example.)

\begin{enumerate}
\item The bottom entries appearing under an equal top unprimed entry are weakly increasing.
\item The bottom entries appearing under an equal top primed entry are weakly decreasing.
\end{enumerate}

We form a pair of hook composition tableaux with the same underlying partition from $M(\omega) = {a_1 \choose b_1} \cdots {a_n \choose b_n}$ by letting $P= \emptyset \leftarrow b_1\cdots b_n$ and recording the growth of $P$ n the hook composition tableau $Q$.  That is, place $a_j$ into the highest available cell in the $k^{th}$ column of $Q$ where $k$ is the index of the column in which a new cell was created upon inserting $b_j$ into $P$.  The fact that $P$ is a hook composition tableau is immediate from Lemma~\ref{insertHCT}.  We must check that $Q$ is a hook composition tableau.

To see that $Q$ is indeed a hook composition tableau, recall Berele and Remmel's insertion procedure from Definition~\ref{def:reminsert} used to prove Theorem~\ref{BRrsk}.  This insertion procedure commutes with the insertion procedure used to construct the hook composition tableau $P$ by Lemma~\ref{lem:commute}, so the column sets in the recording hook composition tableau $Q$ are the same as the column sets of the recording tableau in the Berele-Remmel construction.  Since there is a unique hook composition tableau with each valid collection of column sets, and we know this collection is valid since it can also be obtained from the Berele-Remmel recording tableau.  The recording diagram is this hook composition tableau by construction since the entries were placed into the columns in increasing order just as they are in the map from semistandard hook tableaux to hook composition tableaux.

To reverse the process, begin with the highest occurrence of the largest primed entry $d$ of $Q$.  Let $L$ be the length of the row containing this entry.  Find the smallest entry $e$ in $P$ such that $e$ appears as the last entry of a row of length $L$.  This value $e$ is the value in $P$ corresponding to $d$.  Reverse the insertion steps to determine the sequence of entries which bumped $e$ and find the initial entry that was inserted into a hook composition tableau $P'$ to obtain $P$.  Place this entry under $d$ in the biword and repeat, building the biword from right to left.
\end{proof}

\begin{figure}
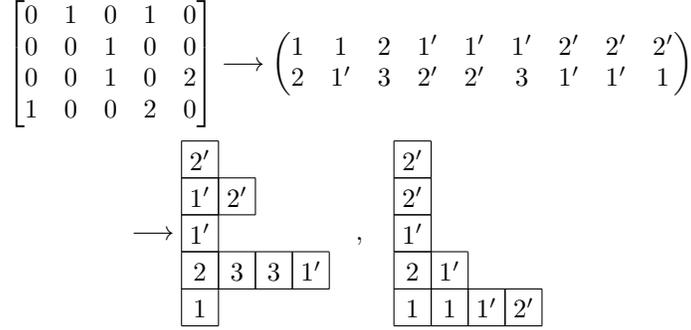

$$\begin{bmatrix}
0 & 1 & 0 & 1 & 0 \\ 0 & 0 & 1 & 0 & 0 \\ 0 & 0 & 1 & 0 & 2 \\ 1 & 0 & 0 & 2 & 0
\end{bmatrix}  \longrightarrow \begin{pmatrix} 1 & 1 & 2 & 1' & 1' & 1' & 2' & 2' & 2' \\ 2 & 1' & 3 & 2' & 2' & 3 & 1' & 1' & 1  \end{pmatrix}$$$$ \longrightarrow \tableau{2' \\ 1' & 2' \\ 1' \\ 2 & 3 & 3 & 1' \\ 1} \quad , \quad \tableau{2' \\ 2' \\ 1' \\ 2 & 1' \\ 1 &  1 & 1' & 2'} \quad $$
\caption{Example of the RSK analogue for $M \in \mathcal{M}(k_1,l_1,k_2,l_2)$}
\label{fig:RSK}
\end{figure}

The following analogue of the Cauchy identity is a corollary to Theorem~\ref{rsk} but we prove it using Equation~\ref{cauchy} and Theorem~\ref{thm:hookSchur}.

\begin{cor}
\begin{align*}
\sum_{\lambda} \sum_{\boldsymbol{\lambda}({\alpha})=\boldsymbol{\lambda}({\beta})=\lambda}& \left( HQ_{\alpha}( X;S) HQ_{\beta}(Y;T) \right)\\ &= \prod_{i,j} (\frac{1}{1-x_i y_j}) \prod_{i,j}(\frac{1}{1 - s_i t_j}) \prod_{i,j} (1+x_i t_j) \prod_{i,j} ( 1 + y_i s_j),
\end{align*} where $\alpha$ and $\beta$ are compositions.
\end{cor}

\begin{proof}
Recall Equation~\ref{cauchy}: \begin{align*}\sum_{\lambda} HS_{\lambda}( X;S)& HS_{\lambda}(Y;T) \\&= \prod_{i,j} (\frac{1}{1-x_i y_j}) \prod_{i,j}(\frac{1}{1 - s_i t_j}) \prod_{i,j} (1+x_i t_j) \prod_{i,j} ( 1 + y_i s_j),\end{align*} which provides a generating function for products of hook Schur functions.  Theorem~\ref{thm:hookSchur} states that: $$HS_{\lambda} (X;Y) = \sum_{\boldsymbol{\lambda}({\alpha})=\lambda} HQ_{\alpha} (X; Y);$$ substitute this refinement of the hook Schurs into the generating function identity to obtain \begin{align*}\sum_{\lambda} \sum_{\boldsymbol{\lambda}({\alpha})=\lambda} HQ_{\alpha} &(X;Y) \sum_{\boldsymbol{\lambda}({\beta})=\lambda} HQ_{\beta} (X;Y)\\& = \prod_{i,j} (\frac{1}{1-x_i y_j}) \prod_{i,j}(\frac{1}{1 - s_i t_j}) \prod_{i,j} (1+x_i t_j) \prod_{i,j} ( 1 + y_i s_j),\end{align*} which reduces to \begin{align*}\sum_{\lambda} \sum_{\boldsymbol{\lambda}({\alpha})=\boldsymbol{\lambda}({\beta})=\lambda}& \left( HQ_{\alpha}( X;S) HQ_{\beta}(Y;T) \right)\\&= \prod_{i,j} (\frac{1}{1-x_i y_j}) \prod_{i,j}(\frac{1}{1 - s_i t_j}) \prod_{i,j} (1+x_i t_j) \prod_{i,j} ( 1 + y_i s_j)\end{align*} when the summations are combined.
\end{proof}

\section{A Littlewood-Richardson rule}\label{sec:LR}

The multiplication of hook Schur functions behaves exactly the same as the multiplication of Schur functions in the sense that the structure constants are the same.  

\begin{theorem}[\cite{Rem84}]
If $$s_{\nu}(X) s_{\mu}(X) = \sum_{\lambda} {g^{\lambda}_{\nu, \mu}} s_{\lambda}(X), $$ then $$HS_{\nu}(X;Y) HS_{\mu}(X; Y) = \sum_{\lambda} {g^{\lambda}_{\nu, \mu}} HS_{\lambda}(X; Y).$$
\end{theorem}
\begin{figure}
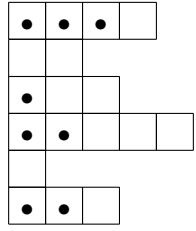

\[\tableau{{\bullet}&{\bullet}&{\bullet}&{}\\{}&{}\\{\bullet}&{}&{}\\{\bullet}&{\bullet}&{}&{}&{}\\{}\\{\bullet}&{\bullet}&{}}\]
\caption{The empty cells indicate $\beta/\alpha$ where $\beta=(3,1,5,3,2,4), \alpha=(2,2,1,3),$ and $\gamma=(2,0,2,1,0,3)$.}\label{fig:gamma}
\end{figure}
We prove an analogous result for products of quasisymmetric hook Schur functions and hook Schur functions. 
\begin{theorem}\label{thm:LRrule}
If \[\mathcal{CS}_\alpha(X) s_\mu(X)=\sum_{\beta:\alpha \subseteq \beta} A_{\alpha\mu}^\beta \mathcal{CS}_\beta(X),\]
then 
\begin{equation}\label{eq:LR}
\mathcal{HQ}_\alpha(X;Y)HS_\mu(X;Y) = \sum_{\beta: \alpha\subseteq \beta} A_{\alpha\mu}^\beta \mathcal{HQ}_\beta (X;Y).
\end{equation}
\end{theorem}
In Theorem \ref{thm:LRrule} the structure constant $A_{\alpha\mu}^\beta$ is the number of {\em Littlewood-Richardson composition tableaux} of shape $\beta/\alpha$ with content $\mu$.  Given compositions $\alpha, \beta$, a skew composition tableau $\beta / \gamma$ is of shape $\beta/\alpha$ provided there $\gamma$ is a weak composition (some parts may be zero) with $\gamma \subseteq \beta$ and $\gamma^+=\alpha$ where $\gamma^+$ is $\gamma$ with all the zero parts removed, as seen in Figure~\ref{fig:gamma}.  A tableau $V$ of shape $\beta/\alpha$ is a {\em Littlewood-Richardson composition tableau} when 

\begin{enumerate}
\item The cells in $\alpha$ are filled with zeros, as is an additional column, column 0, to the left of the first column of $\beta$.  Where there are multiple 0's in the same column, we consider the 0's to be increasing from top to bottom.
\item Each row is weakly increasing from left to right.
\item A triple of cells in rows $i$ and $j$, $i<j$, is a {\em Type A triple} if $\beta_i\leq \beta_j$ with the cells arranged as shown in Figure~\ref{fig:abtriples}, and is a {\em Type B triple} if $\beta_i>\beta_j$ with the cells arranged as shown in Figure~\ref{fig:abtriples}.  All Type A and B triples are {\em inversion triples}, meaning $c\leq b <a$ or $a<c\leq b$.  \begin{figure}
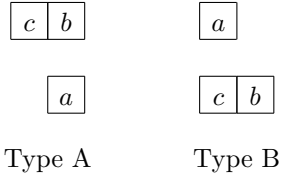

\begin{center}
\begin{tabular}{cc}
${\displaystyle\tableau{c&b\\\\&a}}$ \hspace{.3in}&\hspace{.3in} ${\displaystyle \tableau{a\\\\c&b}}$\\\\
Type A \hspace{.3in}&\hspace{.3in} Type B
\end{tabular}
\end{center}
\caption{The two types of triples between rows $i$ and $j$, $i<j$, in a tableau of shape $\beta$.  When $\beta_i\leq \beta_j$ consider Type A triples.  When $\beta_i>\beta_j$ consider Type B triples.}\label{fig:abtriples}
\end{figure}
 \item The {\em column reading word} $w_{col}(T)$ obtained by reading down each column of $V$ starting with the rightmost column (omitting all 0's) must be a {\em lattice word}, a word $w_1w_2\ldots w_n$ where, for each $i<\ell(\mu)$ and each prefix $w_1w_2\ldots w_j$ with $j\leq n$ the number of $i$'s in the prefix is weakly greater than the number of $i+1$'s in the prefix.
\end{enumerate}

To prove Theorem~\ref{thm:LRrule} we define a map $g:HCT(\alpha)\times SSHT(\mu)\rightarrow \bigcup_{\beta: \alpha\subseteq \beta} (LR(\beta/\alpha,\mu)\times HCT(\beta))$in the following way:
\begin{enumerate}
\item Given $(U,T)\in HCT(\alpha)\times SSHT(\mu)$, construct the tableau $T_\mu$ by filling the Ferrers diagram of $\mu$ with the label $i$ in each cell in row $i$.  
\item Perform $RSK^{-1}(T,T_\mu)$ to obtain a double word $\binom{\sigma}{\tau}$.
\item Then $g((U,T)) = (V,S)$ where $S$ is the hook composition tableau obtained by inserting $\tau$ into $U$ and $V$ is the of shape $\beta/\alpha$ obtained by placing 0's in the cells of $\alpha$ and placing each label $\sigma_i$ from $\sigma$ into $V$ in the location where a new cell was created by the insertion of $\tau_i$ into $U$.
\end{enumerate}

We claim that $V$ is a Littlewood-Richardson composition tableau of shape $\beta/\alpha$ and content $\mu$ and that $g$ is a bijection.  An example of $g$ can be seen in Figure~\ref{fig:bijection}.  Prior to proving this, we first prove some technical lemmas.  The first establishes that in the insertion process (step 3 above), once a row $i$ is longer than some row $j$, $i<j$, then row $i$ remains longer than row $j$ as labels are inserted.  

\begin{figure}
\[U=\tableau{1'&2'&3'&4'\\1'\\2&2&3&1'\\1&1&1'}\quad,\quad T=\tableau{1'\\1'&2'\\2&2&3&1'\\1&1&1&2}\quad,\quad
T_\mu = \tableau{4\\3&3\\2&2&2&2\\1&1&1&1}\]
\bigskip

\[ \rightarrow \left(\begin{array}{ccccccccccc}1&1&1&1&2&2&2&2&3&3&4\\2&2&1'&2'&1&1&3&1'&1&1'&2\end{array}\right)\]
\bigskip

\[\rightarrow V=\tableau{0&0&0&0\\0\\1\\3\\4\\2&2&3\\0&0&0&0&1&1\\0&0&0&1&2&2}\quad,\quad S=\tableau{1'&2'&3'&4'\\1'\\1'\\1'\\1'\\3&1'&2'\\2&2&2&2&3&1'\\1&1&1&1&1&2}
 \]
\caption{Example of $g$}\label{fig:bijection}
\end{figure}

\begin{lemma}\label{lem:rowlength}
Let $T$ be a hook composition tableau of shape $\beta$.  If $\beta_i>\beta_j$ for some $i<j$, then in $T\leftarrow k$, row $i$ is strictly longer than row $j$.  \end{lemma}

\begin{proof}
Let $T$ be a hook composition tableau with row $i$ strictly longer than row $j$ for some $i<j$ and let $k \in \mathcal{A}\cup \mathcal{A}'$.  First note that if $\beta_i>\beta_j+1$, then row $i$ will still be strictly longer than row $j$ in $T\leftarrow k$.  

Thus we assume $\beta_i=\beta_j+1$.  Then there is a triple of cells in $T$ with labels $a,c,\infty$
\[\tableau{a&\infty\\\\&c}\] with $a,\infty$ in row $j$ and $c$ in row $i$.   Since this is a valid triple, it must be that $a>c$ if $c\in \mathcal{A}$ (resp. $a\geq c$ if $c\in \mathcal{A}'$).  Note that if $k \in \mathcal{A}$ (resp. $k \in \mathcal{A}'$) and $k\geq c$ (resp. $k>c$), then the new cell in $T\leftarrow k$ will be placed in a column strictly to the right of $c$.  This is guaranteed since $k$ (or any label $x\geq k$ bumped during the insertion process) must scan the column directly to the right of $c$ from top to bottom.  If $k$ (or $x$) is not placed in a new cell earlier, a new cell will be created in row $i$ directly to the right of the cell containing $c$ and thus row $i$ is still strictly longer than row $j$.  

So, we can assume $k<c$ (resp. $k\leq c$) and suppose, to obtain a contradiction, that the new cell in $T\leftarrow k$ is placed at the end of row $j$ with label $d\geq k$.  Then the triple of cells under consideration is 
\[\tableau{a&d\\\\&c}\quad.\]  
For this triple to be valid, $d>c$ (resp. $d\geq c$) and $k<c<d$ (resp. $k\leq c\leq a<d$).  If $d$ was bumped from a column to the right of the cell containing $c$ then either $d$ would have been placed in a new cell to the right of $c$ in row $i$ unless $d$ were bumped from the column directly right of the cell containing $c$ and a row below $i$.  However, this cannot happen because then there would be a triple of the form \[\tableau{c&\infty\\\\&d}\] where $c<d<\infty$ which contradicts the triple rule.  Thus, $d$ must have been bumped from the column containing $c$ and a row above row $j$.

Since $k<c$ but $d>c$ (resp. $d \ge c$), then in the same column as $c$ there must be an entry $f<c$  (resp. $f\leq c$) that bumps some entry $g$ (possibly $g=d$) with $g\geq c$ (resp. $g>c$).  Thus, prior to inserting $k$ the triples involving $f,g,c$ are arranged 
\[\tableau{\overline{f}&f\\\\\overline{g}&g\\\\ &c}.\]  For $\overline{g}, g,$ and $c$ to form a valid triple we must have $c<\overline{g}\leq g$ (resp. $c\leq \overline{g}<g$).  Then, in $T\leftarrow k$, $f$ replaces $g$, so $c<\overline{g}\leq f$ (resp.$c\leq \overline{g}<f$), but this contradicts $f<c$ (resp. $f\leq c$).  Thus, the new cell created in $T\leftarrow k$ cannot be in row $j$.
\end{proof}

The next lemma provides a convenient way to prove that the column word of a composition tableau $T$ satisfying certain conditions is a lattice word.

\begin{lemma}\label{lem:lattice}
Let $T$ be a tableau of shape $\beta/\alpha$ and content $\mu$ satisfying the type A and B triple conditions.  If the word obtained from $T$ by arranging the entries in each column, starting with the rightmost column, in weakly increasing order is a lattice word, then $w_{col}(T)$ is also a lattice word and $T$ is a Littlewood-Richardson composition tableau.
\end{lemma}

\begin{proof}
Let $T$ be a tableau of shape $\beta/\alpha$ and content $\mu$ satisfying the type A and B triple conditions.  Let $C_k$ be the set of nonzero entries in the $k$th column of $T$ arranged in weakly increasing order.  Then $u=C_mC_{m-1}\cdots C_1$ is the word obtained by arranging the entries in each column of $T$ in weakly increasing order.  Suppose $u$ is a lattice word.  For each $i$ in $w_{col}(T)$ let $i_k$ denote the $k$th occurrence of $i$ from right to left.  Note that for $w_{col}(T)$ to be a lattice word, $i_k$ must appear earlier than $(i+1)_k$ in $w_{col}(T)$ for each $i$ and each $k$. 

Suppose $w_{col}(T)$ is not a lattice word.  Then there exist $i, i+1$ such that $(i+1)_k$ appears earlier in $w_{col}(T)$ than $i_k$ for at least one $k$.  Since $u$ is a lattice word, $i_k$ and $(i+1)_k$ must be in the same column of $T$ with $i_k$ appearing in row $x$ and $(i+1)_k$ appearing in row $y$ where $x<y$.  

Suppose first that $\beta_x\leq \beta_y$.  Then we consider the type A triple
\[\cellsize=5.5ex\tableau{a&{\scriptstyle (i+1)_k}\\\\&i_k}.\]  Without loss of generality we consider $k$ maximal among such pairs of $i_k, (i+1)_k$.  For this triple to satisfy the triple condition, we must have $a=i+1$.  Thus, since each column can have at most one occurrence of each label, $a=(i+1)_{k+1}$.  Since $k$ is maximal, $i_{k+1}$ must appear in the same column as $a$ and above $a$ as shown
\[\cellsize=5.5ex\tableau{{\scriptstyle i_{(k+1)}} & b\\\\a&{\scriptstyle (i+1)_k} \\\\c&i_k}.\]  
If the row $z$ containing $i_{k+1}$ has $\beta_z\geq \beta_y$, then $\beta_z\geq \beta_x$ and the type A triple formed by $i_{k+1}, b, i_k$ does not satisfy the triple condition.  If $\beta_z<\beta_y$ but $\beta_z\geq \beta_x$, then, as before there is a type A triple violation.  Finally, if $\beta_z<\beta_x$, then the type B triple $i_{k+1}, c, i_k$ does not satisfy the triple condition.

Next suppose $\beta_x>\beta_y$.  Then we consider the type B triple 
\[\cellsize=5.5ex\tableau{{\scriptstyle (i+1)_k}\\\\i_k & c}.\] Without loss of generality we consider $k$ minimal among such pairs of $i_k, (i+1)_k$.  For the triple to satisfy the triple conditions, we must have $c=i_{k-1}$.  Then $(i+1)_{k-1}$ must appear in the column below $i_{k-1}$ due to the lattice condition and the minimality of $k$.  Suppose $(i+1)_{k-1}$ appears in row $z$, $z<x$.  Then we have 
\[\cellsize=5.5ex\tableau{{\scriptstyle{(i+1)}_k}&a\\\\i_k&c\\\\b&{\scriptstyle i+1 }}.\]
If $\beta_z>\beta_y$ then $(i+1)_k, b, (i+1)_{k-1}$ violate the triple condition.  If $\beta_z\leq \beta_y$ then $(i+1)_k,a,(i+1)_{k-1}$ violate the type A triple condition.  

Therefore $w_{col}(T)$ must be a lattice word and hence $T$ is a Littlewood-Richardson composition tableau.
\end{proof}

We are now ready to prove the Littlewood-Richardson rule for quasisymmetric $(k,l)$-hook Schur functions.

\begin{proof}({\em of Theorem~\ref{thm:LRrule}})

Let $(U,T) \in HCT(\alpha) \times SSHT(\mu)$ and let $g((U,T)) = (V,S)$.  Note that the rows of $V$ are weakly increasing since new cells in $S$ are created at the right end of an existing row or in a new row of length one. 
We now consider the type $A$ and $B$ triples in $V$.  First, suppose there is a type $A$ triple of cells where the entries are $a,b,c$ as in Figure~\ref{fig:abtriples}.  Suppose the triple is not an inversion triple, that is, $c\leq a<b$.  Then either $c=a=0$ or $a$ would have been placed before $b$, contradicting Lemma~\ref{lem:rowlength}.  Thus all type $A$ triples in $V$ must be inversion triples.  Next, suppose there is a type $B$ triple of cells with the entries $a,b,c$ as in Figure~\ref{fig:abtriples} and suppose $c<a\leq b$, that is, the triple is not an inversion triple.  Then, since $a>0$, the cells containing $c$ and $a$ are not in column $0$ of $V$.  If $a<b$, then the cell containing $a$ must have been added to $V$ prior to the cell containing $b$ since the entries are added in increasing order, contradicting Lemma~\ref{lem:rowlength}.  If $a=b$, then $a$ would have be placed to the right of $b$, contradicting our assumption on the positions of $a$ and $b$.

By Lemma~\ref{lem:commute} the bijection $f$ and insertion commute.  Thus, the column entries of the Littlewood-Richardson tableau $L$ obtained by inserting the entries of $\tau$ into $f^{-1}(U)$ while keeping track of created cells by placing the corresponding entries of $\sigma$ into the Littlewood-Richardson tableau are the same as the entries in each column of $V$.  Note that if we let $C_i$ denote the set of entries in column $i$ of $V$ placed in weakly increasing order, the word $u=C_jC_{j-1}\cdots C_1$ (arranging column entries from the righmost column to the left) is a lattice word since it is the column reading word of $L$.   By Lemma~\ref{lem:lattice}, $w_{col}(V)$ is a lattice word and $V$ is a Littlewood-Richardson composition tableau.

Given a pair $(V,S)$ where $V\in LR(\beta/\alpha,\mu)$ and $S\in HCT(\beta)$ we compute $g^{-1}((V,S))$ using the following steps:
\begin{enumerate}
\item In $V$ let $m$ denote the largest nonzero label and remove the rightmost cell with label $m$ from $V$.  
\item In $S$, remove the corresponding cell by taking the label $a$ in the cell and scanning $S$ in reverse reading order starting from the cell being removed.  Then $a$ will \emph{evict} the first entry $b$ such that $a>b$ if $a\in \mathcal{A}$ (resp. $a\geq b$ if $a\in \mathcal{A}'$) and $a\leq c$ (resp. $a<c$) where $c$ is the cell immediately right of $b$.  Continue scanning now with $b$.  The algorithm terminates when all cells prior to the cell containing $a$ in reading order have been scanned.  We denote the HCT resulting from the removal of the cell containing $a$ by $a\leftarrow S$.
\item The final entry evicted from $S$ will be placed into the bottom row of a double word, while $m$ is the corresponding entry in the top row.   
\item Continue removing cells from $V$ one at a time, each time removing a cell containing the largest nonzero label from the rightmost cell with that label.  Remove the corresponding cells one at a time from $S$ as before and prepend the results in the double word.  Once all nonzero entries in $V$ are deleted, and the empty rows are removed, the remaining cells of $S$ form a composition tableau of shape $\alpha$ that we will call $U$.   We will show that $U$ is a HCT.
\item Take the resulting double word $\binom{\pi}{\sigma}$ and use the traditional RSK algorithm to insert $\sigma$ into an (initially) empty tableau while keeping track of cells added with the corresponding labels in $\pi$.  The result of inserting $\sigma$ will be called $T$. 
\end{enumerate}

We show that the result of each cell removal in $S$ is a HCT.  To do this, note that the only way to take a diagram satisfying the triple rule and remove an entry to create a diagram that doesn't satisfy the triple rule is to remove an entry from position $c$ such that the entry to the left of $c$ is less than or equal to (resp. less than) an entry in a lower row of the column containing $c$.  We must therefore show that if $a<b$ are in the same column, with $a$ at the end of its row and $a$ in a higher row than $b$ that the cell containing $a$ will not be removed prior to the cell containing $b$.  Suppose such a pair $a<b$ exists in $S$ with $a$ in cell $(j,k)$ and $b$ in cell $(i,k)$, $i<j$ with $\beta_i=k$ and $\beta_j=k$.  In $V$ let $e$ denote the entry in cell $(j,k)$ and $f$ denote the entry in cell $(i,k)$.  Then in order for the cell $(j,k)$ to be removed from $S$ prior to cell $(i,k)$, we have $e>f$.  Since $V$ satisfies all the triple conditions, we know $k>1$ and the entry of $(j,k-1)$ in $V$ is nonzero.  If all the entries in row $j$ are greater than all the entries in row $i$, then there is a type A triple consisting of the entries in $(j,0),(j,1), (i,1)$.  Since the entry in $(j,1)$ is greater than the entry in $(i,1)$ but the entry in $(i,1)$ is greater than the entry in $(j,0)$, there is a violation of the type A triple rules.  Suppose now that there is some point in rows $j$ and $i$ where 
\[\tableau{\cdots &g&h & \cdots & e\\\\\cdots & x&y&\cdots & f}\] with $g<x$ and $h>y$.  Then $g<y<h$ which violates the type A triple rule.  Thus removing each cell in $S$ according to the above rules results in a valid HCT.  The result of removing all cells corresponding to nonzero entries of $V$ results in a HCT of shape $\alpha$.

It remains to be shown that when Definition~\ref{def:reminsert} is used to insert the bottom row of the double word $\binom{\pi}{\sigma}$ from left to right into an empty tableau the result is a semistandard hook tableau and the corresponding entries in the top row are placed into the corresponding cells in a second semistandard tableau with the result being $T_\mu$.  This is equivalent to showing that the subwords 
\[\begin{pmatrix} i&i& \cdots & i \\ \sigma_{i_1}& \sigma_{i_2}&\cdots &\sigma_{i_{\mu_i}}\end{pmatrix}\] have the property that $\sigma_{i_j}\leq \sigma_{i_{j+1}}$ for $1\leq j < \mu_i$ and that $\sigma_{i_{\mu_{i}}}>\sigma_{(i+1)_1}$ for each $i$ if $\sigma_{i_{\mu_{i}}}\in \mathcal{A}$ (resp. $\sigma_{i_{\mu_i}}\geq \sigma_{(i+1)_1}$ if $\sigma_{i_{\mu_i}}\in \mathcal{A}'$).

Suppose the cells containing $x$ and $y$ in $S$ correspond to cells containing $i$ in $V$ with $x$ appearing first in column reading order.  Then the cell in $S$ containing $x$ will be removed prior to the cell containing $y$ in reading order.  Note that the row containing $y$ must be strictly shorter than the row containing $x$ since there is at most one $i$ per column in $V$ and further, if the row containing $y$ has length $\beta_j$ and the row containing $x$ has length $\beta_m$ with $m<j$, then $\beta_j<\beta_m-1$, otherwise the type B triple rule will be violated in $V$.  When the cell containing $x$ is removed from $S$, a value $c\leq x$ is removed from $S$ and in the new HCT $S'$, the length of the $m$th row is $\beta_m-1$.  Suppose the entry of cell $(m,\beta_m-1)$ in $S'$ is $d$.  Then $d\leq x$.  If $y\geq x$, then $y$ will evict $d$ when the cell containing $y$ is removed since $x$ has already been removed and there is therefore nothing to the right of $y$.  The process  continues and eventually the entry scanning in reverse reading order continuing the removal process is smaller than $c$ or, if $e$ was the label that evicted $c$ from $S$, eventually the scanning entry $g$ will scan the row containing $e$ and will evict the entry $h$ left of $e$.  Thus $h\leq c$ and the entry removed from $S'$ will be weakly less than $c$.

To prove  $\sigma_{i_{\mu_{i}}}>\sigma_{i+1_1}$ for each $i$ if $\sigma_{i_{\mu_{i}}}\in \mathcal{A}$ (resp. $\sigma_{i_{\mu_i}}\geq \sigma_{i+1_1}$ if $\sigma_{i_{\mu_i}}\in \mathcal{A}'$), we first note that the rightmost $i$ must appear before the leftmost $i+1$ in column reading order due to the requirement that $w_{col}(V)$ is a lattice word.

Let $S'$ be the HCT obtained by removing all entries corresponding to entries in $V$ which are greater than $i+1$ as well as all but the leftmost $i+1$.  Let $v_i$ and $v_{i+1}$ denote the cells in $V$ containing the rightmost $i$ and the leftmost $i+1$ respectively.  Note that in $S'$ there are corresponding cells which we call $u_i$ and $u_{i+1}$ respectively.  We denote by $x$ and $y$ the labels in $S$ in cells $u_i$ and $u_{i+1}$ respectively.  

We remove $y$ first because $y$ corresponds to $i+1$ and $x$ corresponds to $i$.  Let $y \ge c_1 \ge c_2 \ge \cdots \ge c_k$ be the eviction route for the removal of $y$ and let $d_0 \ge d_1 \ge d_2 \ge \cdots \ge d_m$ be the eviction route for the removal of the entry $d_0$ in cell $u_i$ of $y \leftarrow S'$.  (Note that if $c_j \in \mathcal{A}$ for some $j$, all inequalities following $c_j$ are strict.  Similarly for the bumping path for $x$.)  Since $y$ appears after $x$ in column reading order, the removal of $y$ will scan the cell containing $x$.  Let $c_j$ be the entry in the eviction route which scans $x$.  If $x$ is not evicted and $c_j \in \mathcal{A}$, then $c_j \le x$.  (If $x$ is not evicted and $c_j \in \mathcal{A}'$, then $c_j < x$.)  If $x$ is evicted, then $x=c_{j+1}$ and $c_j$ is the entry in cell $u_i$ in $y \leftarrow S'$.  Either way, the entry remaining in the cell $u_i$ of $y \leftarrow S'$ is greater than or equal to the entry that scans the next entry during the removal of $y$.  Moreover, if the entry that scans the next entry during the removal of $y$ is in $\mathcal{A}'$, then the entry remaining in the cell $u_i$ of $S'$ is strictly greater than the entry that scans the next entry during the removal of $y$.

Assume that the removal of $y$ is complete and consider the removal of the entry $d_0$ in the cell $u_i$ of $y \leftarrow S'$.  Let $b$ be the first entry evicted during this removal and let $s_0$ be the cell containing $b$.  That means $d_0 > b$ if $d_0 \in \mathcal{A}$ (respectively $d_0 \ge b$ if $d_0 \in \mathcal{A}'$).  The entry $c_n$ (where $n$ might be $j$ or $j+1$) which left cell $u_i$ during the removal of $y$ was less than or equal to the entry $d_0$ (strictly less if $c_n \in \mathcal{A}'$), and hence some entry $c_l$ of the eviction route less than or equal to $d_0$ (strictly less than $d_0$ if $c_n \in \mathcal{A}'$) scanned the cell $s_0$ during the removal of $y$.  Either $c_l$ evicted the entry in $s_0$ or $c_l$ was weakly smaller (strictly smaller if $c_l \in \mathcal{A}'$) than the entry in $s_0$.  Either way, the entry leaving $s_0$ during the removal of $y$ was weakly less than (strictly less than if $c_l \in \mathcal{A}'$) the entry remaining there, and hence the entry $b$ is weakly greater (strictly greater if $c_l \in \mathcal{A}'$) than the next entry it sees from the escape route of $y$.  Repeated iteration of this argument shows that the entries in the escape route for $x$ remain weakly greater (and strictly greater if they are in $\mathcal{A}$) than the entries in the escape route for $y$, meaning the output for removing $x$ is weakly greater (strictly greater if the output is in $\mathcal{A}$) than the output for removing $y$.  
 
 The final step in our proof is to show that our eviction process is in fact the inverse of insertion.  This is a two-step process.  First we prove that if $S'$ is the result of the insertion $S \leftarrow k$ then applying the eviction process to the new entry created will produce the HCT $S$.  The second step is to prove that if $m$ is removed from an HCT $S'$ via eviction, then inserting the final evicted entry $e$ back into the resulting HCT $S$ via our insertion process produces the HCT $S'$.  We accomplish these two steps by proving that the eviction routes and bumping paths are the same.  We use a modified version of the proof that appeared in~\cite{AHM16}; the only difference in our approach is our consideration of the double alphabet $\mathcal{A} \cup \mathcal{A}'$.
 
 Assume first that $S'$ is the result of the insertion $S \leftarrow k$ and apply the removal process to the entry $m$ in the new cell created during this process.  We now prove that the bumping path $(c_1,d_1), (c_2,d_2), \dots , (c_g,d_g)$ for $S \leftarrow k$ is identical to the eviction route $(c'_1,d'_1), (c'_2, d'_2), \dots , (c'_h,g'_h)$ for the removal of $m$ from $S'$; that is, $(c'_1,d'_1)=(c_g,d_g), (c'_2,d'_2)=(c_{g-1},d_{g-1}), \dots , (c'_h,d'_h)=(c_1,d_1)$.  Assume not to obtain a contradiction.  
 
 Assume first that $(c'_i,d'_i)$ is the first cell in reverse reading order in the eviction route but not the bumping path.  This means that during the insertion of $k$ into $S$, the entry scanning this cell did not bump $S(c'_i,d'_i)$, but did bump the next entry in the eviction route, $S(c'_{i-1},d'_{i-1})$.  We know that the entry $e$ contained in $S'(c_{i-1}',d_{i-1}')$ after insertion is the entry scanning $S'(c_i',d_i')$ during the removal of $m$, since the eviction route agrees with the bumping path up to this point (in reverse reading order).  Furthermore, the bumping path does not include any cells between $(c'_i,d'_i)$ and $(c'_{i-1},d'_{i-1})$ so in fact $e$ is the entry which scanned the cell $(c'_i,d'_i)$ during insertion.  Since $e$ evicts $S'(c_i',d_i')$ during removal, we have $S'(c'_i , d_i'-1) \le e$ if $e \in \mathcal{A}$ (resp. $S'(c'_i, d_i'-1) < e$ if $e \in \mathcal{A'}$) and $e > S'(c'_i,d'_i)$ if $e \in \mathcal{A}$ (resp. $e \ge S'(c'_i,d'_i)$ if $e \in \mathcal{A}'$).  Since the entry in $S'(c'_i,d'_i-1)$ after removal is the same as the entry in $S(c'_i,d'_i-1)$ (by equality of bumping path and eviction route up to $(c'_i,d'_i))$, we have $e \ge S(c'_i,d'_i-1)$ if $e \in \mathcal{A}$ (resp. $e > S(c'_i,d'_i-1)$ if $e \in \mathcal{A}'$).  Since $e$ does not bump $S(c'_i,d'_i)$, we must have $e \ge S'(c'_i,d'_i)$ if $e \in \mathcal{A}$ (resp. $e > S'(c'_i,d'_i)$ if $e \in \mathcal{A}'$), a contradiction.  Therefore if the eviction route and the bumping path differ, their first difference (in reverse reading order) involves a cell which is in the bumping path but not the eviction route.
 
Assume next that $(c_i,d_i)$ is the first cell in reverse reading order in the bumping path of $k$, but not in the eviction route of $m$.  Let $e=S(c_i,d_i)$ be the entry bumped during the insertion of $k$ in $S$.  From this we know that $S(c_i,d_i-1)<e$ since $e$ is bumped during insertion, and $e\leq S(c_i,d_i+1)= S'(c_i,d_i+1)$ if $e\in \mathcal{A}$ (resp. $e<S'(c_i,d_i+1)$ if $e\in \mathcal{A}'$) since $e$ was bumped and at most one entry from each row can be bumped during insertion.  Thus, during eviction in $S'$, $e$ is the label that scans cell $(c_i,d_i)$.  Since $e$ does not evict $S'(c_i,d_i)$ it must be that either $e\leq S'(c_i,d_i)$ or $e>S'(c_i,d_i+1)$ if $e\in \mathcal{A}$ (resp. $e<S'(c_i,d_i)$ or $e\geq S'(c_i,d_i+1)$ if $e\in \mathcal{A}'$).  Note that since $S'(c_i,d_i)$ bumped $e$ from cell $(c_i,d_i)$ during the insertion process,  $S'(c_i,d_i)<e$ if $e\in \mathcal{A}$ (resp. $S'(c_i,d_i)\leq e$ if $e\in \mathcal{A}'$) and from the preceding discussion, we already know that $S'(c_i,d_i+1)\leq e$ if $e\in \mathcal{A}$ (resp. $S'(c_i,d_i+1)<e$ if $e\in \mathcal{A}'$), so $e$ should evict $S'(c_i,d_i)$ and we have a contradiction.

We have shown that there cannot be a position at which the bumping path for $S \leftarrow k$ differs from the eviction route for removing $m$ from $S'$, and hence the bumping path and eviction route must be identical.  Therefore $(m\leftarrow S')=S$, as desired.

To go the other way, assume that $m$ is removed from a hook composition tableau $S'$ with final removed entry $k$ to obtain an HCT $S$, and then $k$ is inserted into $S$.  We will prove that $(S\leftarrow k)=S'$ by showing the eviction route for the removal of $m$ is identical to the bumping path for the insertion of $k$.  This proof is similar to the above argument so we use broader strokes here to provide intuition supplementing the precision in the argument above.

If the eviction route and bumping path are identical then the proof is complete.  Therefore we may assume that they differ at some cell.  Either this cell is in the eviction route and not the bumping path or this cell is in the bumping path and not the eviction route.

Assume first that $(c_i',d_i')$ is the first cell in reading order that is in the eviction route but not the bumping path.  The next entry it evicts is in the bumping path, as are all the remaining entries evicted as the removal comes to a conclusion.  Therefore the entry $e$ to scan $(c_i',d_i')$ in $S$ is the entry that was evicted from $(c_i',d_i')$ of $S'$ during removal.  This entry $e$ scanning cell $(c_i',d_i')$ during insertion must be smaller than (resp. smaller than or equal to) $S'(c_i',d_i')$ and was situated in this cell before removal; hence it satisfies all conditions necessary to bump this entry, a contradiction.

Assume next that $(c_i,d_i)$ is the first cell in reading order in the bumping path but not the eviction route.  Since the bumping path and eviction route agree up to this point, this means that we are bumping an entry which was not evicted, hence the entry $e$ doing this bumping passed by this cell without evicting its entry during removal.  But since $e$ bumps the entry in this cell during insertion, we must have $e$ greater than or equal to (resp. greater than) the entry immediately to its left.  But then $e$ would have evicted the entry in cell $(c_i,d_i)$ of since $e$ is greater than (resp. greater than or equal to) this entry, a contradiction.

We have shown that there cannot be a first position at which the bumping path for inserting $k$ into $S$ differs from the eviction route for the removal of $m$ from $S'$, and hence the bumping path and eviction route must be identical.  Therefore inserting $k$ into $S$ produces $S'$, as desired.

We have shown that when we insert and then remove we get the same bumping and eviction route sequence.  We have also shown that if we remove and then insert we get the same bumping path and eviction route.  It follows that insertion and eviction are inverses of each other.
\end{proof}

\section{Acknowledgements}

The authors would like to thank Jeff Remmel for suggesting this problem.

\newcommand{\etalchar}[1]{$^{#1}$}

\end{document}